\documentclass[a4paper,11pt]{amsart}

\usepackage[utf8]{inputenc}
\usepackage[T1]{fontenc}
\usepackage{amsfonts}
\usepackage{amsthm}
\usepackage{amsmath}
\usepackage[english]{babel}
\usepackage[all]{xy}
\usepackage{graphicx}
\usepackage{amscd}
\usepackage{latexsym}
\usepackage{hyperref}
\usepackage{mathrsfs}
\usepackage{enumerate}

\usepackage{color}

\usepackage{calc}

\theoremstyle{definition}
 \newtheorem{defi}{Definition}[section]

\theoremstyle{remark}
 \newtheorem{remark}[defi]{Remark}

\theoremstyle{plain}
\newtheorem{theo}[defi]{Theorem}
\newtheorem{prop}[defi]{Proposition}
\newtheorem{cor}[defi]{Corollary}
\newtheorem{lemma}[defi]{Lemma}

\newtheorem*{acknow}{Acknowledgments}

\newcommand{\zz}{\mathbb{Z}}

\newcommand{\rr}{\mathbb{R}}
\newcommand{\cc}{\mathbb{C}}

\newcommand{\Z}[1]{\zz_{#1}}

\newcommand{\abs}[1]{\left\vert #1 \right\vert}

\newcommand{\A}{\mathscr{A}}

\newcommand{\Mcal}{\mathcal{M}}
\newcommand{\Fcal}{\mathcal{F}}
\newcommand{\Gcal}{\mathcal{G}}
\newcommand{\Hcal}{\mathcal{H}}

\newcommand{\M}{\mathscr{M}}
\newcommand{\Mg}{\mathscr{M}^\mathfrak{g}}

\newcommand{\G}{\mathscr{G}}
\newcommand{\Gc}{\mathscr{G}^c}

\newcommand{\N}{\mathscr{N}}
\newcommand{\Nc}{\mathscr{N}^c}

\newcommand{\I}{\mathcal{I}}

\newcommand{\Symp}{\textbf{Symp}}
\newcommand{\Cob}{\textbf{Cob}}  
\newcommand{\Cobelem}{\textbf{Cobelem}}

\newcommand{\arnaque}{\begin{flushright}
$\Box$
\end{flushright}}

\newcommand{\MW}{Manolescu and Woodward}
\newcommand{\WW}{Wehrheim and Woodward}
\newcommand{\LL}{Lekili and Lipyanskiy}

\newcommand{\glag}{generalized Lagrangian correspondence}
\newcommand{\hsi}{symplectic instanton homology}
\newcommand{\hSi}{Symplectic instanton homology}

\newcommand{\gint}{generalized intersection point}

\newcommand{\qfh}{quilted Floer homology}
\newcommand{\fft}{Floer field theory}

\title[HSI: naturality, and maps from cobordisms]{Symplectic Instanton Homology: naturality, and maps from cobordisms}
\author[Guillem Cazassus]{Guillem Cazassus}

\address{Department of Mathematics, Indiana University, Bloomington, IN 47405}
\email{gcazassu@iu.edu}

\begin{document}

\begin{abstract}We prove that Manolescu and Woodward's Symplectic Instanton homology, and its twisted versions, are natural; and define maps associated to four dimensional cobordisms within this theory.

This allows one to define representations of the mapping class group, the fundamental group and the first cohomology group with $\Z{2}$ coefficients of a 3-manifold. We also provide a geometric interpretation of the maps appearing in the long exact sequence for symplectic instanton homology, together with vanishing criterions.
\end{abstract}

\maketitle
\tableofcontents

\section{Introduction}\label{sec:intro}

\hSi\ is a 3-manifold invariant that was introduced by \MW\ in \cite{MW}. In order to establish a Dehn surgery long exact sequence, we also introduced a twisted version in \cite{surgery}.  This twisted version is associated to a 3-manifold $Y$ with a class $c\in H_1(Y;\Z{2})$, and corresponds to \MW's invariant if $c=0$.

These groups are defined after choosing a specific Heegaard splitting of the 3-manifold (or more generally a Cerf decomposition), and it turns  out that two such decompositions yield isomorphic groups.

\vspace{.3cm}
\paragraph{\underline{Naturality}} Strictly speaking, the only invariant obtained a priori from such a procedure is the isomorphism type of the group. Such a construction is said to be natural if it permits to define an actual group. In other words, given any two such decompositions, one has to be able to find a \emph{canonical} isomorphism between the two groups obtained.
This is a central issue regarding several constructions, including:

\begin{itemize}
\item being able to define invariants that take the form of homology classes in these groups, such as for example classes associated to contact structures in  Heegaard-Floer theory.

\item defining maps between such groups, in particular maps associated to four-dimensional cobordisms. (Since such maps will be constructed handle by handle, naturality will be particularly important since one cannot compose maps that are defined only up to isomorphism.)

\item defining representations of the mapping class group, and (since a basepoint will be involved) of the fundamental group.
\end{itemize}

In Heegaard-Floer theory, naturality has been established by Juh\'asz, Thurston and Zemke in \cite{JuhaszThurston}. Among other things, it follows from their work that the group $\widehat{HF}$ depends on the choice of the pointed Heegaard diagram only through the basepoint. Likewise, the HSI groups also depend on the choice of a basepoint. However, the groups will not just depend on a class $c\in H_1(Y;\Z{2})$, but rather on an $SO(3)$-principal bundle $P$ over $Y$, for which $c$ is Poincare dual to $w_2(P)$.
Analogous constructions also appeared in \cite{horton} in a slightly different setting.

\begin{theo}\label{th:naturality} Let $(Y,z)$ be a pointed closed oriented 3-manifold, and $P$ an $SO(3)$-bundle over $Y$. Then, the group $HSI(Y,P,z; \Z{2})$ constructed in \cite[section~3]{surgery} (where it was denoted $HSI(Y,c,z; \Z{2})$), with coefficients in $\Z{2}$, are natural invariants of $(Y,P,z)$ in the following sense. Given two Cerf decompositions of $W$, the blow-up of $Y$ at $z$, seen as a "cobordism with vertical boundary", one can associate two "generalized Lagrangian correspondences" $\underline{L}$ and  $\underline{L}'$. Then, given a sequence of Cerf moves relating the two decompositions, the isomorphism $\Phi\colon HF(\underline{L}) \to HF(\underline{L}')$ constructed in section~\ref{sec:naturality} is independent in the choice of the sequence of Cerf moves.
\end{theo}


\begin{remark}In this paper, all $HSI$ groups will be with $\Z{2}$-coefficients. We will drop $\Z{2}$ from the notations. Analogous results should hold with $\zz$ coefficients.
\end{remark}

 Given a diffeomorphism $\varphi\colon Y\to Y'$, one obtains from naturality an induced map  
 \[
 F_{\varphi}\colon HSI(Y,\varphi^* P,z) \to HSI(Y',P,\varphi(z)).
 \]
As a consequence, when $P$ is trivial,  the mapping class group of $(Y,z)$ acts on $HSI(Y,z)$, and the fundamental group $\pi_1(Y,z)$ acts on $HSI(Y,z)$ by moving the basepoint. We will describe these actions more explicitly in section~\ref{ssec:mcgrep}.


Furthermore, $\pi_0(Aut (P))\simeq H^1(Y; \Z{2})$ acts on $HSI(Y,P,z)$ in a nontrivial way, as we will see for Lens spaces in section~\ref{ssec:action_H1}. For this reason, these groups are not natural invariants of $(Y,w_2(P),z)$ as opposed to what we initially thought. However, when $w_2(P)=0$, since there is a preferred trivial bundle $P = Y\times SO(3)$, one can talk about $HSI(Y,z)$ as a natural invariant, even though $H^1(Y; \Z{2})$ still acts nontrivially on it in general.

\vspace{.3cm}

\paragraph{\underline{Maps from cobordisms}} Let $W$ be a compact connected oriented smooth 4-cobordism from $Y$ to $Y'$ (both closed connected 3-manifolds), $P$ an $SO(3)$-bundle over $W$ and $\gamma\colon [0,1]\to W$ a path from  $Y$ to $Y'$ connecting two basepoints $z$ and $z'$. We will define a map 
\[ 
F_{W,P,\gamma}\colon HSI(Y,P_{|Y},z) \to  HSI(Y',P_{|Y'},z')  
.\]
Here are our three main motivations.

First,  to provide a geometric interpretation of the maps appearing in the surgery exact sequence of \cite{surgery}:

\begin{theo}\label{th:vaguesurgery}(see Theorem~\ref{thinterpfleches} and Corollary~\ref{corinterpfleches} for precise statements) With suitable choices of principal bundles, the maps from cobordisms associated to 2-handle attachments provide a long exact sequence for a Dehn surgery triad.
 \end{theo}

Second, to  obtain geometric vanishing and nonvanishing criterions for these maps appearing in the long exact sequence. The blow-up formula in Corollary~\ref{eclatement} is an example of such a vanishing criterion, and one can hope to have analogs of the adjunction formulas as in Heegaard-Floer theory.

Last, to define invariants for 4-manifolds with boundary. In analogy with instanton homology, one should be able to define invariants similar with relative Donaldson polynomials that take values in the HSI groups associated to the 3-dimensional boundary. These maps should be seen as a first step towards this goal, and could correspond to the constant part of such polynomials, see subsection~\ref{ssec:bigpic} for more details.

In section~\ref{sec:naturality}, we briefly review   \WW's Floer field theory and how \hsi\ fits into this framework. We then prove naturality, using an argument similar with \cite{JuhaszThurston}. In section~\ref{sec:cob} we construct the maps from cobordisms, handle-by-handle as in \cite{OSholotri} but using \WW's quilt theory, prove that these are well-defined and give some of their properties.

\begin{acknow}
Section 3 was part of my PhD thesis, I would like to thank my advisors Paolo Ghiggini and Michel Boileau for their constant support. I also would like to thank Chris Woodward for pointing out to me the importance of naturality in the construction of the cobordism maps, and Nate Bottman, Paul Kirk and Yank{\i} Lekili for helpful conversations.
\end{acknow}

\section{Naturality}\label{sec:naturality}

The $HSI$ groups are built after choosing a Heegaard splitting (or a Cerf decomposition) of $Y$. \MW\ proved in \cite[sec.~6.3]{MW} that two choices of Heegaard splittings yield isomorphic groups. We prove that an isomorphism can be chosen in a canonical way, once a basepoint of $Y$ is fixed.

\subsection{\fft, \hsi\ and geometric composition}

We briefly review the construction of (twisted) \hsi\ within \WW's \fft\ framework. Then we recall the two approaches for the geometric composition isomorphisms: \WW's strip shrinking argument \cite{WWcompo}, and \LL's "Y-maps", which furnish the isomorphism in \MW's proof of stabilization invariance of HSI. For further details we refer to \cite{surgery} and \cite{LekiliLipyanskiy,LLcorr}, and to \cite{MW} for the original construction. See also \cite{WWfft} and \cite[Remark~3.5.8]{Wehrheimphilo}.

We then show that the two last approaches are equivalent in the setting of \hsi, using Bottman's results.

\begin{defi}[The category $\Cob_{2+1}$] \label{def:cobcat}
We will call \emph{category  of $SO(3)$-cobordisms with vertical boundaries}, and will denote it $\Cob_{2+1}$, the category whose:

\begin{itemize}

\item  objects are 4-tuples $(\Sigma,p,P, A^\partial)$, where $\Sigma$ is a compact connected oriented surface, with connected boundary,  $p\colon \rr/\zz \rightarrow \partial\Sigma$ is a diffeomorphism (parametrization), and $P$ is a trivial $SO(3)$-bundle over $\Sigma$, equipped with a flat connexion $A^\partial$ with no holonomy on $\partial \Sigma$ (which is the same as fixing a trivialization of $P$ on $\partial \Sigma$, up to an overall constant gauge transformation).

\item  morphisms from $(\Sigma_0,p_0,P_0, A^\partial_0)$ to $(\Sigma_1,p_1,P_1, A^\partial_1)$ are equivalence classes  of tuples $(W, \pi_{0},  \pi_{_1}, p,  P, \varphi_0, \varphi_1,  A^\partial)$, where $W$ is a compact oriented 3-manifold with boundary, $\pi_{\Sigma_0}$,  $\pi_{\Sigma_1}$ and $p$ are embeddings of $\Sigma_0$, $\Sigma_1$ and $\rr/\zz \times [0,1]$ into $\partial W$, the first  reversing the orientation, the two others preserving it. These are such that:
\begin{itemize}
\item $\partial W = \pi_{\Sigma_0}(\Sigma_0)\cup  \pi_{\Sigma_1}(\Sigma_1)\cup  p(\rr/\zz \times [0,1])$,
\item $\pi_{\Sigma_0}(\Sigma_0)$  and $\pi_{\Sigma_1}(\Sigma_1)$ are disjoint,
\item for $i=0,1$, $p(s,i) = \pi_{\Sigma_i}(p_i(s))$, and 
\[
\pi_{\Sigma_i}(\Sigma_i)\cap p(\rr/\zz \times [0,1])  = \pi_{\Sigma_i}(p_i(\rr/\zz)) = p(\rr/\zz \times \lbrace i\rbrace)
,\]
\end{itemize}

We will refer to   $p(\rr/\zz \times [0,1]) $ as the vertical part of $\partial W$, and will denote it $\partial^{vert} W$.

 $P$ is an $SO(3)$-bundle over $W$, the $\varphi_i$ are bundle isomorphisms between $P_i$ and $P$ covering $\pi_i$, and $A^\partial$ is a flat connexion on $P_{|\partial^{vert}W}$ with no holonomy on the $\rr/\zz$ direction. We assume that $A^\partial$ pulls back to $A^\partial_i$ via $\varphi_i$.

Two such tuples $(W, P, \cdots)$ and $(W', P', \cdots ')$ are equivalent and will be identified if there exists a diffeomorphism $\phi \colon W \rightarrow W'$ compatible with  the embeddings, and a bundle isomorphism $\varphi\colon P\to P'$ covering $\phi$ and compatible with the $\varphi_i$'s and the connexions on the vertical boundaries. 

\item  composition of morphisms consists in gluing along the embeddings and bundle isomorphisms.
\end{itemize}

\end{defi}

To keep notations short we will simply write $(W,P)$ from $(\Sigma_0, P_0)$ to $(\Sigma_1, P_1)$, or just $W$ from $\Sigma_0$ to $\Sigma_1$ when $P = W\times SO(3)$. Notice that in \cite{surgery} we used a slightly different category (using cohomology classes instead of bundles).

\begin{defi}[The category $\Symp$]  \label{def:sympcat}
We will call $\Symp$ the following category:  

$\bullet$ Its objects are tuples $(M, \omega , \tilde{\omega}  , R, \tilde{J})$ satisfying conditions $(i)$, $(ii)$, $(iii)$, $(iv)$, $(v)$, $(x)$, $(xi)$ and $(xii)$ of  \cite[Assumption 2.5]{MW}, namely:
\begin{enumerate}
\item[$(i)$] $(M,\omega)$ is a compact symplectic manifold.

\item[$(ii)$] $\tilde{\omega}$ is a closed  2-form on $M$.

\item[$(iii)$] The degeneracy locus $R\subset M$ of $\tilde{\omega}$ is a symplectic hypersurface  for $\omega$.

\item[$(iv)$] $\tilde{\omega}$ is $\frac{1}{4}$-monotone, that is $\left[ \tilde{\omega} \right] = \frac{1}{4} c_1(TM) \in H^2(M; \rr)$.

\item[$(v)$] The restrictions of $\tilde{\omega}$ and $\omega$ to $M\setminus R$ define the same cohomology class in  $H^2(M\setminus R; \rr)$.

\item[$(x)$] The minimal Chern number $N_{M\setminus R}$ with respect to $\omega$ is a positive multiple  of 4, so that the minimal Maslov number   $N = 2N_{M\setminus R}$ is a positive  multiple of $8$.

\item[$(xi)$] $\tilde{J}$ is an $\omega$-compa\-tible almost complex structure  on $M$, $\tilde{\omega}$-compa\-tible on $M\setminus R$, and such that $R$ is an  almost complex hypersurface for $\tilde{J}$.

\item[$(xii)$] Every index zero $\tilde{J}$-holomorphic sphere in $M$, necessarily contained in $R$ by monotonicity, has an intersection number with  $R$ equal to a negative multiple  of 2.

\end{enumerate}

$\bullet$ The set of morphisms between two objects consists in strings of elementary morphisms $\underline{L} = (L_{01}, L_{12}, \cdots ) $, modulo an equivalence relation:
\begin{itemize}
\item The elementary morphisms are  correspondences $L_{i(i+1)} \subset  M_i^- \times M_{i+1}$ that are Lagrangian for the monotone forms $\tilde{\omega}_i$, simply con\-nected, $(R_i, R_{i+1})$-compatible (in the sense that $L_{i(i+1)}$ intersects $R_i \times M_{i+1}$ and $M_i \times R_{i+1}$ transversally, and these two intersections coincide), such that $L_{i(i+1)} \setminus \left( R_i \times R_{i+1} \right)$ is spin, and such that every pseudo-holomor\-phic disc of $M_i^- \times M_{i+1}$ with boundary in $L_{i(i+1)}$ and zero area has an intersection number with $(R_i, R_{i+1})$ equal to a positive multiple  of $-2$.
\item The equivalence relation on strings of morphisms is generated by the following identification:  $(L_{01},\cdots ,  L_{(i-1)i},L_{i(i+1)}, \cdots )$ is identified with $(L_{01},\cdots ,  L_{(i-1)i}\circ  L_{i(i+1)}, \cdots )$ whenever the composition of $L_{(i-1)i}$ and $L_{i(i+1)}$ is embedded, simply connected, $(R_{i-1} , R_{i+1} )$-compatible, spin outside $R_{i-1}\times R_{i+1}$, satisfies the above  hypothesis concerning  pseudo-holomorphic discs, and also the following one: every  quilted pseudo-holomorphic cylinder of zero area  and with seam conditions in $L_{(i-1)i}$,   $L_{i(i+1)}$ and $L_{(i-1)i}\circ  L_{i(i+1)}$ has an intersection number with $(R_{i-1} , R_i, R_{i+1})$ smaller than $-2$.
\end{itemize}
\end{defi}

A (2+1)-\fft\ is then a functor $\textbf{F}\colon \Cob_{2+1} \to \Symp$. The functor involved in the definition of \hsi\ will associate to an object $(\Sigma, P, \cdots)$ a moduli space $\Nc(\Sigma,P)$, and to a morphism $(W,P, \cdots)$ a generalized Lagrangian correspondence $\underline{L}(W,P)$, about which we recall the definitions now. 

\begin{defi}\emph{(Extended moduli spaces)}\label{def:extmod}
\begin{itemize}
\item (Connection constant along the boundary) we will say that a connexion $A\in\A(\Sigma, P)$ is constant along $\partial \Sigma$ if, after having trivialized $P_{|\partial \Sigma}$ so that $A^\partial$ is the horizontal connexion, the connexion $A$, viewed as a $\mathfrak{g}$-valued one-form on $\Sigma$, can be written as $A_{|\partial \Sigma} = \theta ds$, for some constant element $\theta \in \mathfrak{g}$ (which value depends on the trivialization of $P_{|\partial \Sigma}$).

\item (Extended moduli space associated to a surface) 

Define $\A_F^\mathfrak{g}(\Sigma, P) \subset \A(\Sigma, P)$ to be the space of connexions that are flat on $\Sigma$, and constant on $\partial \Sigma$.

This subspace is acted on by the group 
\[
\Gc (\Sigma, P) = \left\lbrace \varphi \in Aut(P)\ |\ \varphi _{|P_{|\partial \Sigma}} = Id \right\rbrace
.\]
Denote $\Gc_0 (\Sigma, P) \subset \Gc (\Sigma, P)$ the connected component of the identity, and define the extended moduli space as the quotient $\Mg (\Sigma, P) = \A_F^\mathfrak{g} (\Sigma, P) /\Gc_0 (\Sigma, P)$.

This space carries a closed 2-form $\omega$ defined by:
\[ \omega_{[A]}([\alpha],[\beta]) = \int_{\Sigma'} \langle \alpha\wedge\beta \rangle   ,\]
with $ [A]\in \Mg (\Sigma, P) $ and $\alpha,\beta$ representing tangent vectors at $[A]$ of $\Mg (\Sigma, p)$, namely $d_A$-closed ${ad}(P)$-valued  1-forms, constant near $\partial \Sigma$.


 Furthermore it is acted on by $\G^{const}_0 / \Gc_0 \simeq SO(3)$ in a Hamiltonian fashion, with $\G^{const}_0$ standing for the gauge transformations acting by multiplication by a constant element over $\partial \Sigma$. After trivializing $P$ over $\partial \Sigma$ and identifying $\mathfrak{g}$ with its dual, the moment map is given by the element $\theta\in \mathfrak{su(2)}$ such that $A_{| \partial \Sigma} = \theta ds$.

\item  Denote $\N(\Sigma, P)$ the subset of $\Mg(\Sigma, P)$ consisting in equivalence classes of connections for which $\abs{\theta} < \pi \sqrt{2}$. The form $\omega$ is symplectic on $\N(\Sigma, P)$.

\item The space $\Nc(\Sigma, P)$ is defined as a symplectic cutting of $\Mg (\Sigma, P)$ for the function $\abs{\theta}$ at $\pi \sqrt{2}$:
\[
\Nc(\Sigma, P) =  \Mg(\Sigma,P)_{\leq \pi \sqrt{2}}  
.\]
It can be seen as a compactification of the subset $\N(\Sigma, P)$ by gluing the codimension 2 submanifold $R =  \lbrace |\theta| = \pi \sqrt{2} \rbrace / U(1) $:
\[
\Nc(\Sigma, P) =   \N(\Sigma,P) \cup R 
.\]
\end{itemize}

\end{defi}

\begin{remark}(Identification with Huebschmann-Jeffrey moduli space) When $P$ is trivial, a trivialization of $P$ induces a canonical identification of $\Mg(\Sigma,P)$ with the $SU(2)$-analogue $\Mg(\Sigma)$,  defined in \cite[Def. 2.1]{jeffrey} and \cite{huebschmann}. This is because the two spaces  $\A_F^\mathfrak{g}(\Sigma, P)$ and  $\A_F^\mathfrak{g}(\Sigma)$ are identified, and  $\Gc_0 (\Sigma, P) \simeq \Gc (\Sigma)/ \Z{2}$, with $\Gc (\Sigma)$ consisting in $SU(2)$-gauge transformations. Therefore all the results in \cite{jeffrey} and \cite{MW} about $\Mg(\Sigma)$ still hold for $\Mg(\Sigma,P)$.
\end{remark}

\begin{defi}[Moduli space and (generalized) Lagrangian correspondences associated to a 3-cobordism]  
\label{def:lagcor}
Let $(W, P)$ be an $SO(3)$-cobordism with vertical boundary from $(\Sigma_0, P_0)$ to $(\Sigma_1, P_1)$. 

\begin{itemize}
\item (Moduli space associated to an $SO(3)$-cobordism with vertical boundary). Define $\mathscr{A}_F^\mathfrak{g}(W,P)$ analogously as for a surface: flat connexions that are constant near $\partial^{vert}W$.

\item (Correspondence associated to a cobordism with vertical boundary.) Let  
\[ 
L(W,P) \subset \Mg (\Sigma_0, P_0)^- \times \Mg (\Sigma_1, P_1)
\]
be the correspondence consisting in the pairs of connexions that extend flatly  to $P$:
\[
L(W,P) = \{([A_{|\Sigma_0} ], [A_{|\Sigma_1} ])\ |\ A \in \Mg(W,P)\}
.\]

\item ($L^c(W,c)$). Define 
\[
L^c(W,p,c)\subset \Nc(\Sigma_0,p_0)^- \times \Nc(\Sigma_1,p_1)
\]
 as the closure  of $ L(W,P) \cap \left( \N (\Sigma_0,P_0)^- \times \N (\Sigma_1,P_1) \right)$. If $W$ is elementary in the sense of Cerf theory (it admits a Morse function with at most one critical point, and whose restriction to the vertical part of the boundary has no critical points), this is a morphism of $\Symp$.

\item  ($\underline{L}(W,P)$) Assume now that $W$ is not elementary anymore, and choose a Cerf decomposition $\underline{W} = (W_0, \cdots , W_k)$, where the $W_i$ are elementary. Define the following \glag:
\[
 \underline{L}(W,P)=  \left( L^c(W_0,P_0) , \cdots , L^c(W_k,P_k) \right) 
 ,\]
 with $P_i$ the restriction of $P$ to $W_i$.
\end{itemize}
\end{defi}

These satisfy the assumptions of definition \ref{def:sympcat}, and the \glag\  $\underline{L}(W,P)$, as a morphism in $\Symp$, is independent from the decompositions of $W$, see \cite[Prop.~3.18, Th.~3.22 ]{surgery}.

To a \glag\ $\underline{L}$ from $pt$ to $pt$ can be associated a homology group called \qfh. We recall roughly its definition, with $\Z{2}$-coefficients (see remark~\ref{rem:perturb_continuations} below, and \cite{WWqfc} for the complete construction):

\begin{defi}($\I(\underline{L})$,  $CF(\underline{L})$, $\partial$,  $HF(\underline{L})$).  \label{def:quiltfloerhomol}
\begin{itemize}
\item Define the set of \gint s as: 
\[
 \I(\underline{L}) =\lbrace (x_0, \cdots , x_k) ~|~ \forall i, (x_i,x_{i+1}) \in L_{i(i+1)}   \rbrace, 
 \]
We say that $\underline{L}$  has \emph{transverse intersection} if $L_{01}\times L_{23}\times \cdots$ and $L_{12}\times L_{34}\times \cdots$ intersect transversely in $M_0\times M_1\times M_2\times \cdots$. In this case, $\I(\underline{L})$ is a finite set (in general one should use Hamiltonian perturbations).
\item Let  $CF(\underline{L})$ be the free $\Z{2}$-module generated by  $\I(\underline{L})$.
\[CF(\underline{L}) = \bigoplus_{\underline{x}\in \I(\underline{L})} \Z{2} \underline{x},\]

\item The differential $\partial$ counts pseudo-holomorphic quilted strips.  
\item The \qfh\ $HF(\underline{L})$ is defined as the homology of the chain complex $(CF(\underline{L}),\partial)$. It is independent from the Hamiltonian perturbations and almost complex structures, and is relatively $\Z{8}$-graded.

\end{itemize}

\end{defi}

\begin{remark}\label{rem:perturb_continuations}
In this construction one has to fix auxiliary data, namely almost complex structures and Hamiltonian perturbations that are regular. 
Given two such choices $(\underline{J}, \underline{H})$ and $(\underline{J}', \underline{H}')$, one can define continuation maps as in \cite[Prop.~5.3.2]{WWqfc} to identify $HF(\underline{L}; \underline{J}, \underline{H})$ and $HF(\underline{L}; \underline{J}', \underline{H}')$ in a canonical way, so that $HF(\underline{L})$ is well-defined as a group. 
\end{remark}

A fundamental result of \WW's theory is that these groups are well-behaved for geometric composition, provided the following holds:

\begin{defi} \label{def:embcomp}
\begin{itemize}
\item (Geometric composition) Let $M_0$, $M_1$, $M_2$ be three symplectic manifolds, and $L_{01} \subset M_0^-\times M_1$, $L_{12} \subset M_1^-\times M_2$ be Lagrangian correspondences. The \emph{geometric composition} of $L_{01}$ with $L_{12}$ is the subset:
\[L_{01} \circ L_{12} = \pi_{02}(L_{01}\times M_2 \cap M_0\times L_{12}),\]
where $\pi_{02}$ denotes the projection
\[\pi_{02}\colon M_0\times M_1\times M_2 \rightarrow M_0\times  M_2 .\]

\item (Embedded geometric composition) A geometric composition  $L_{01} \circ L_{12}$ is said to be \emph{embedded} when:

\begin{itemize}
\item $L_{01} \times M_2$ and $M_0 \times L_{12}$ intersect transversally,

\item $\pi_{02}$ induces an embedding of $L_{01}\times M_2 \cap M_0\times L_{12}$ in $M_0\times  M_2$.
\end{itemize}

\item (Embedded$^+$ geometric composition) Assume now that $M_0$, $M_1$ and $M_2$ are objects of $\Symp$, with hypersurfaces $R_0$, $R_1$, $R_2$, and that $L_{01} \subset M_0^-\times M_1$ and $L_{12} \subset M_1^-\times M_2$ are morphisms in $\Symp$. We say that $L_{01} \circ L_{12} $ is \emph{embedded$^+$} if it is embedded, simply connected, $(R_0, R_2)$-compatible, and such that the  intersection number of every pseudo-holomorphic quilted cylinder  with  $(R_{i-1} , R_i,  R_{i+1})$ is smaller than $-2$.
\end{itemize}
\end{defi}

The following isomorphism has been proven in \cite{WWcompo} and then proven in a different setting by \LL\ using a more geometric construction, see also \cite[3.5.8]{Wehrheimphilo}. This last construction has been extended by \MW\ to the setting of the category $\Symp$:

\begin{theo}(\cite[Theorem 6.7]{MW})\label{th:LLcomp} Let $\underline{L}$ be a generalized Lagrangian correspondence as before. Assume moreover that for some index $i$, the geometric composition $L_{i-1,i} \circ  L_{i, i + 1}$ is embedded$^+$, then $HF (\underline{L})$ is canonically isomorphic to $HF (\cdots L_{i-1,i} \circ L_{i, i + 1} \cdots )$.
\end{theo}

Recall also the isomorphisms built by \LL\ involved in the proof of theorem~\ref{th:LLcomp}, called $Y$-maps: these maps 
\[ \begin{cases}
\Phi\colon HF(\cdots , L_{(i-1)i}, L_{i(i+1)}, \cdots) \to HF(\cdots , L_{(i-1)i}\circ L_{i(i+1)}, \cdots)  \\ 
 \Psi\colon HF(\cdots , L_{(i-1)i}\circ L_{i(i+1)}, \cdots)  \to HF(\cdots , L_{(i-1)i}, L_{i(i+1)}, \cdots) 
\end{cases} \]
are defined by counting quilted holomorphic strips (cylinders) as drawn in figure~\ref{fig:Ymaps}.

\begin{figure}[!h]
    \centering
    \def\svgwidth{.80\textwidth}
    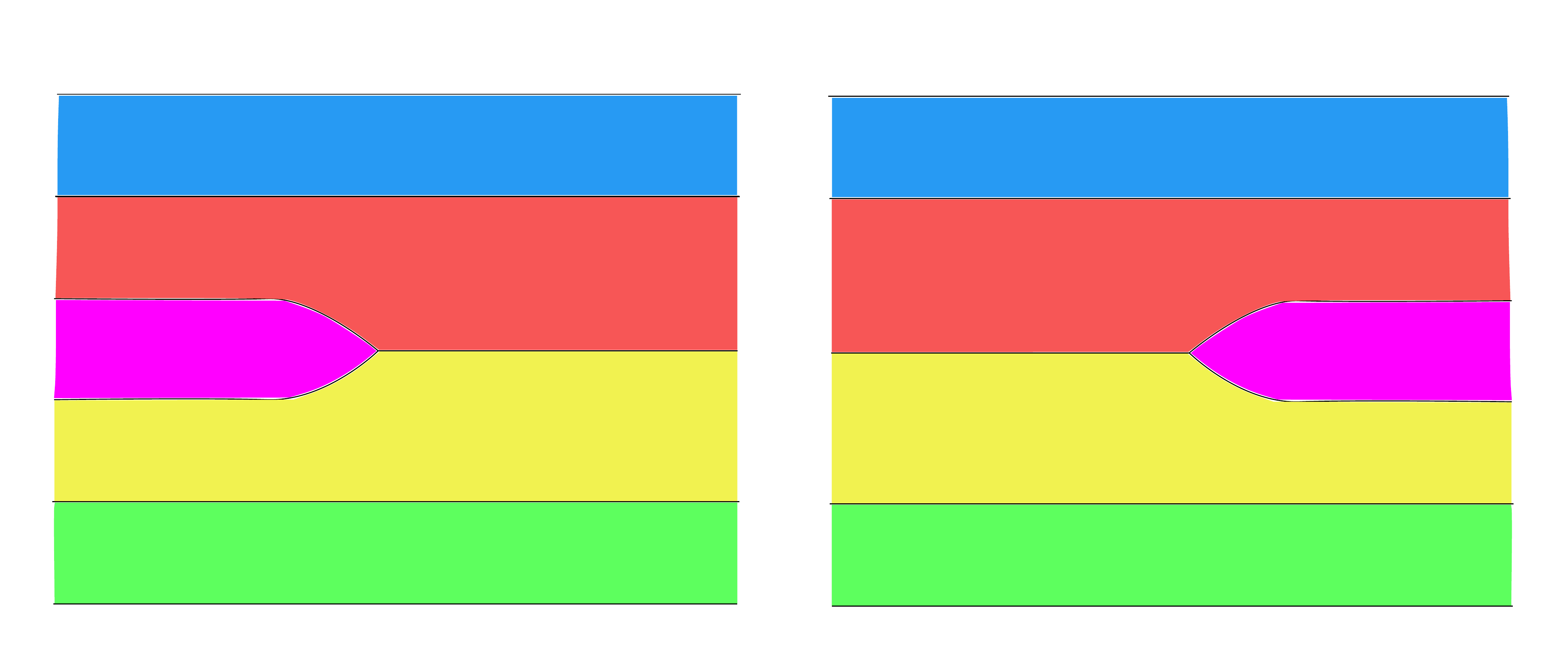
      \caption{}
      \label{fig:Ymaps}
\end{figure}

\begin{remark}\label{rem:Ymapsnilp}In order to prove that $\Phi$ and $\Psi$ are isomorphisms, \LL\ prove that these are inverses up to a nilpotent, i.e. there exists two nilpotent endomorphisms $N_1$ and $N_2$ such that $\Phi \circ \Psi = Id + N_1$ and $ \Psi\circ \Phi = Id + N_2$. In our particular case we will prove that $\Phi$ and $\Psi$ are actually inverses (which is one of the requirements for naturality).
\end{remark}

 
 
The original construction for the geometric composition isomorphism is due to \WW. They first identify the generators of the chain complexes  by the one-to-one correspondence \[p\colon \I(\cdots, L_{(i-1)i} , L_{i(i+1)}, \cdots ) \to \I(\cdots, L_{(i-1)i} \circ L_{i(i+1)}, \cdots )\]
induced by the projection forgetting the $M_i$ coordinate  (which is one-to-one  since the composition is embedded). And then prove that the differentials agree. For proving this they use a strip shrinking argument: they deform the quilted strip associated with $ HF(\cdots , L_{(i-1)i}\circ L_{i(i+1)}, \cdots)$ by  letting the width of the strip corresponding to $M_i$ tend to zero.

\begin{prop}\label{prop:Ymaps_vs_strip_shrinking}
In the setting of $\Symp$, the isomorphisms obtained by \LL's construction coincide with the isomorphisms given by identifying the generators as in \WW's construction.
\end{prop}

\begin{proof}

We prove the proposition for the map $\Phi$, the proof for $\Psi$ is analogous.

We can apply the strip shrinking procedure to the quilted  Y-surfaces: let $(\underline{S}_t)_{t\in (0,1]}$ be the family of quilted surfaces with $\underline{S}_1$ being the $Y$-surface of figure~\ref{fig:Ymaps}, and $\underline{S}_t$ obtained by shrinking the width of the patch associated with $M_i$, until zero when $t \to 0$.

Define  for any generalized intersection points

\begin{align*}
\underline{x} &\in \I(\cdots , L_{(i-1)i}, L_{i(i+1)}, \cdots), \\
\underline{y} &\in \I(\cdots , L_{(i-1)i}\circ L_{i(i+1)}, \cdots)
\end{align*} and for any $t\in (0,1]$,  the moduli space $\mathcal{M}_{t}(\underline{x},\underline{y})$ of index zero pseudo-holomorphic quilts with domain $\underline{S}_t$, limits $\underline{x}, \underline{y}$ at the ends,  and exponential decay at the puncture. Let $\mathcal{M}^{par}(\underline{x},\underline{y})$ denote the  parametrized moduli space

\[\mathcal{M}^{par}(\underline{x},\underline{y}) = \bigcup_{0<t\leq 1}{\mathcal{M}_{t}(\underline{x},\underline{y})} \]

For regular choices of almost complex structures and Hamiltonian perturbations (see  \cite[Th.~3.1.7]{McDuSal}), this is a smooth $1$-dimensional manifold  with boundary, and its boundary $\partial \mathcal{M}^{par}(\underline{x},\underline{y})  = \mathcal{M}_{1}(\underline{x},\underline{y})$ corresponds to the moduli space involved in the definition of $C\Phi$, the chain-level map inducing $\Phi$.

We aim at compactifying $\mathcal{M}^{par}(\underline{x},\underline{y})$. First, by standard Gromov compactness and the assumptions on the category $\Symp$ (definition~\ref{def:sympcat}) we know that for each $\epsilon >0$,   the space
\[\mathcal{M}^{par}_{\geq \epsilon}(\underline{x},\underline{y}) = \bigcup_{\epsilon \leq t\leq 1}{\mathcal{M}_{t}(\underline{x},\underline{y})} \] 
can be compactified by adding broken trajectories at each end,  namely by glueing the following space:

\[ \bigcup_{\tilde{\underline{x}}}{ \Mcal(\underline{x},\tilde{\underline{x}}) \times \Mcal^{par}_{\geq\epsilon, -1} (\tilde{\underline{x}},\underline{y})  } \cup \bigcup_{\tilde{\underline{y}}}{ \Mcal^{par}_{\geq\epsilon, -1}(\underline{x},\tilde{\underline{y}}) \times \Mcal(\tilde{\underline{y}},\underline{y}})  , \]
where  $\Mcal(\underline{x},\tilde{\underline{x}})$ and $\Mcal(\tilde{\underline{y}},\underline{y})$ stand for the moduli spaces of index 1 quilted strips (modulo translations), involved respectively in the definition of the differentials of $CF(\cdots , L_{(i-1)i}, L_{i(i+1)}, \cdots)$ and $CF(\cdots , L_{(i-1)i}, L_{i(i+1)}, \cdots)$, and $\Mcal^{par}_{\geq\epsilon, -1} (\tilde{\underline{x}},\underline{y})  $ stands for the parametrized moduli space of index $-1$ $Y$-quilts, which is generically zero dimensional.

We now want to understand what can be the limit of a sequence of quilted curves $\underline{u}_i \in \Mcal_{t_i}(\underline{x},\underline{y})$  of index 0, or $\underline{u}_i \in \Mcal_{t_i, -1}(\underline{x},\underline{y})$ ( of index -1) for which $t_i \to 0$. According to Bottman-Wehrheim, a (possibly squashed) figure eight bubble may form:

If energy concentrates near a point where the two seams corresponding to $L_{(i-1)i}$ and $L_{i(i+1)}$ come together, then after rescaling and passing to a subsequence, one obtains either 
\begin{itemize}
\item a figure eight bubble (in  the case when the concentration speed is commensurate to the  shrinking speed $t_i$). This is a triple of pseudo holomorphic maps
\begin{align*}
u_{i-1} &\colon (-\infty, 0] \times \rr \to M_{i-1},\\
u_{i} &\colon [0, 1] \times \rr \to M_{i}, \\
u_{i+1} &\colon [1, \infty) \times \rr \to M_{i+1} \\
\end{align*}
that satisfy seam conditons in $L_{(i-1)i}$ and $L_{i(i+1)}$.

\item  a "squashed" figure eight bubble (in the case where energy concentrates slower).This is a triple of  maps
\begin{align*}
u_{i-1} &\colon (-\infty, 0] \times \rr \to M_{i-1},\\
u_{i} &\colon  \rr \to M_{i}, \\
u_{i+1} &\colon [0, \infty) \times \rr \to M_{i+1} \\
\end{align*}
with $u_{i-1}$ and $u_{i+1}$ pseudo holomorphic, and that satisfy seam conditons in $L_{(i-1)i}$ and $L_{i(i+1)}$.
\end{itemize}

Bottman's removal of singularity theorem \cite[Th~2.2]{Bottman} states that in the first case, such a triple extends continuously to a quilted 2-sphere. For the squashed case that follows from removal of singularity for disc bubbling: the maps $u_{i-1}$ and $u_{i+1}$ can be "folded" to a map from the half-plane to $ M_{i-1}\times M_{i+1}$, with boundary condition in $ L_{(i-1)i} \circ L_{i(i+1)}$.

We now recall  \MW's argument \cite[Lemma~6.11]{MW} that permits to rule out such above mentioned bubbling penomenon. First, every such bubble must have zero area, otherwise, by monotonicity and the assumption on the minimal Maslov index, the principal component of the limit would have an index too small to exist generically. Hence such bubbles have zero area, and are therefore contained in the zero locus $\underline{R}$ of the symplectic forms. Moreover, the kernel of the symplectic form along $\underline{R}$ coincides with the vertical spaces of an $S^2$-fibration: each patch is then contained in such spheres, and all these patches (eventually squashed) glue to a single map from the sphere to such a fiber. The glued map is orientation preserving, hence if it is nonconstant, it is a branched covering of the sphere, and its intersection number with $\underline{R}$ is a positive multiple of the intersection number of $R_0$ with the fiber $F_0$, which is equal to -2, see \cite[Lemma~4.11]{MW}.

By a result of Cieliebak and Mohnke \cite[Prop~6.9]{CieliebakMohnke},  the principal component of the limit,  generically intersects $\underline{R}$ transversely.  From this fact and the fact that each bubble has intersection smaller than -2, it follows that the total number of intersection points  is strictly greater than the number of bubbles. Hence, there exists transverse intersection points on which no bubbles are attached, which is impossible for a limit of curves disjoint from $\underline{R}$.


Therefore, no bubbling can happen for the limit of curves $u_i$: up to passing to a subsequence, it should converge to a quilted strip of index zero, with limits $\underline{x}$ and $\underline{y}$. Call $\Mcal_0(\underline{x}, \underline{y})$ their moduli space. On the one hand, $\Mcal^{par}_{ -1} (\tilde{\underline{x}},\underline{y})$ is compact (i.e. finite), and $\Mcal^{par}(\underline{x}, \underline{y}) $ is compactified to a closed 1-manifold with boundary, the boundary being: 

\[ \Mcal_0(\underline{x}, \underline{y}) \cup \Mcal_1(\underline{x}, \underline{y}) \cup \bigcup_{\tilde{\underline{x}}}{ \Mcal(\underline{x},\tilde{\underline{x}}) \times \Mcal^{par}_{ -1} (\tilde{\underline{x}},\underline{y})  } \cup \bigcup_{\tilde{\underline{y}}}{ \Mcal^{par}_{ -1}(\underline{x},\tilde{\underline{y}}) \times \Mcal(\tilde{\underline{y}},\underline{y}})  , \]
This implies that the moduli spaces $\Mcal_0(\underline{x}, \underline{y})$ and $\Mcal_1(\underline{x}, \underline{y})$ define the same map  up to homotopy, and in particular induce the same map in homology. Now, it is a standard fact that $ \Mcal_0(\underline{x}, \underline{y})$ is empty if $\underline{y} \neq p(\underline{x})$ (recall that $p$ stands for the projection forgetting the $M_i$ coordinate), and consists in a single point if $\underline{y}= p(\underline{x})$: this is because it is a zero dimensional moduli space endowed with an $\rr$-action: its points are $\rr$-invariant, and therefore constant. So the chain map induced by the $\Mcal_0(\underline{x}, \underline{y})$'s corresponds to \WW's isomorphism. Therefore \LL's $Y$-maps are homotopic to \WW's isomorphisms.

%

\end{proof}

The following lemma will play an important role later, and says that the order of compositions of Lagrangian correspondences do not matter at the level of quilted Floer homology:

\begin{lemma}\label{lem:commutativity_compo}

(A) Let $M_0$, $M_1$, $M_2$ and $M_3$ be objects in $\Symp$, for $i=0, 1,2$ $L_{i(i+1)}\subset M_i^-\times M_{i+1}$ be elementary morphisms in $\Symp$, and $\underline{L}\in hom(pt, M_0)$ and  $\underline{L}'\in hom( M_3, pt)$ be generalized Lagrangian correspondences:
\[
\xymatrix{ pt\ar[r]^{\underline{L}} &M_0 \ar[r]^{L_{01}} &M_1 \ar[r]^{L_{12}} & M_2 \ar[r]^{L_{23}} &M_3 \ar[r]^{\underline{L}'} & pt .}  
\] 

 Assume that:
\begin{itemize}
\item $L_{01}\circ L_{12}$ is embedded$^+$ in the sense of definition~\ref{def:embcomp},
\item $(L_{01}\circ L_{12}) \circ L_{23}$ is embedded$^+$,
\item $ L_{12} \circ L_{23}$ is embedded$^+$,
\item $L_{01}\circ ( L_{12} \circ L_{23})$ is embedded$^+$.
\end{itemize}
Then the following diagram commutes, where each arrow is a composition isomorphism:
\[
\xymatrix{ HF( \underline{L}, L_{01},  L_{12}, L_{23} , \underline{L}'    ) \ar[r]^{i}\ar[d]^{k} & HF( \underline{L}, L_{01}\circ  L_{12}, L_{23} , \underline{L}'    )\ar[d]^{j} \\
HF( \underline{L}, L_{01},  L_{12} \circ L_{23} , \underline{L}' )\ar[r]^{l} & HF( \underline{L}, L_{01}\circ  L_{12} \circ L_{23} , \underline{L}' ) .}  
\] 

(B) Let $M_0$, $M_1$, $M_2$, $M_3$, $M_4$ and $M_5$ be objects in $\Symp$, for $i=0, 1,3,4$ $L_{i(i+1)}\subset M_i^-\times M_{i+1}$ be elementary morphisms in $\Symp$, and $\underline{L}\in hom(pt, M_0)$, $\underline{L}'\in hom( M_2, M_3)$ and  $\underline{L}''\in hom( M_5, pt)$ be generalized Lagrangian correspondences:
\[
\xymatrix{ pt\ar[r]^{\underline{L}} &M_0 \ar[r]^{L_{01}} &M_1 \ar[r]^{L_{12}} & M_2 \ar[r]^{\underline{L}'} &M_3 \ar[r]^{L_{}34}  &M_4 \ar[r]^{L_{45}} &M_5 \ar[r]^{\underline{L}''}& pt .}  
\] 

 Assume that:
\begin{itemize}
\item $L_{01}\circ L_{12}$ is embedded$^+$
\item $L_{34}\circ L_{45}$ is embedded$^+$
\end{itemize}
Then the following diagram commutes, where each arrow is a composition isomorphism:
\[
\xymatrix{ HF( \underline{L}, L_{01},  L_{12}, \underline{L}',    L_{34}, L_{45},  \underline{L}'' ) \ar[r]^{i}\ar[d]^{k} & HF( \underline{L}, L_{01}\circ  L_{12}, \underline{L}',    L_{34}, L_{45},  \underline{L}''   )\ar[d]^{j} \\
HF( \underline{L}, L_{01},  L_{12}, \underline{L}',    L_{34}\circ L_{45},  \underline{L}'' )\ar[r]^{l} & HF( \underline{L}, L_{01}\circ  L_{12}, \underline{L}',    L_{34}\circ L_{45},  \underline{L}'' ) .}  
\] 
\end{lemma}

\begin{proof} In the case when the generalized intersection is transverse on the nose, then the isomorphism is obvious at the chain level, if one uses the strip shrinking approach to the isomorphisms. In the intersection is not transverse, one has to introduce perturbations on each symplectic manifolds appearing in the sequence, but since the compositions are assumed to be embedded, one can take the perturbations on the manifolds where composition is taken to be trivial (namely $M_1$ and $M_2$ in case $(A)$, and $M_1$ and $M_4$ in case $(B)$). Indeed, transversality for the sequence $ \underline{L}, L_{01}\circ  L_{12} \circ L_{23} , \underline{L}'$ in case $(A)$ (resp. for $ \underline{L}, L_{01}\circ  L_{12}, \underline{L}',    L_{34}\circ L_{45},  \underline{L}'' $ in case $(B)$) implies transversality for the other longer sequences. Commutativity of the diagram follows.


\end{proof}

\subsection{Naturality via Cerf theory}\label{ssec:natu_Cerf}

Let $Y$ be a closed connected oriented 3-mani\-fold, $\widehat{P}$ an $SO(3)$-bundle over $Y$, and $z\in Y$ a basepoint. Define $W$ as the real oriented blow-up of $Y$ at $z$, and fix an embedding of $S^1\times [0,1] \to \partial W$. $\widehat{P}$ pulls back to a bundle $P$ over $W$, and since $P_{|\partial W} \simeq P_z \times \partial W$, $P$ comes equipped with a flat connexion $A^\partial$ on $\partial W$, making $(W,P)$ to a morphism of $\Cob_{2+1}$ from the disc to itself. Let $f$ and $f'$ denote two Cerf decompositions of $W$ (i.e. morphisms of the category $\Cobelem$, in the notations of \cite{surgery}) giving rise to two \glag s $\underline{L}(f)$ and $\underline{L}(f')$ from the point to itself, and thus two groups $HF(\underline{L}(f))$ and $HF(\underline{L}(f'))$.

From Cerf theory one knows that $f$ and $f'$ can be joined by a sequence $f_0 = f$, $f_1$, $f_2$, $\cdots$, $f_k = f'$ of decompositions, where $f_{i+1}$ is obtained from $f_i$ by a Cerf move. By \cite[Th.~3.22]{surgery} we know that $\underline{L}(f_i)$ and $\underline{L}(f_{i+1})$ differ from one (or two) embedded geometric composition, hence there exists an isomorphism $\Theta_i \colon HF(\underline{L}(f_i)) \to HF(\underline{L}(f_{i+1})) $, seen either as a $Y$-map, or a \WW's isomorphism.
The composition of all the $\Theta_i$ yields then an isomorphism from $HF(\underline{L}(f))$ to $HF(\underline{L}(f'))$, and one then needs to check that it doesn't depend on the choice of the intermediate decompositions $\lbrace f_i \rbrace$. Without loss of generality, assume for now on that $f=f'$, and let $\Theta = \Theta_{k-1}\circ \cdots \circ \Theta_1\circ \Theta_0$. It then suffices to check that  $\Theta = Id$.

Assume first, for simplicity, that the Lagrangian correspondences in $\underline{L}(f)$ are in general position, so that $\I(\underline{L}(f))$ is a finite set transversely cut out. At the chain level, it suffices to check that the permutation of $\I(\underline{L}(f))$ is the identity. One can see that by remarking that all the  $\I(\underline{L}(f_i))$ can be canonically identified with a common space independent of the  decompositions: the moduli space of framed flat connexions on $\widehat{P}$. Call $\I$ this common moduli space, $\varphi_i\colon \I(\underline{L}(f_i)) \to \I$ the identification, and $p_i\colon \I(\underline{L}(f_i))\to \I(\underline{L}(f_{i+1}))$ the identification given by the projections, the $\varphi_i$ commute with the $p_{i}$ ($ \varphi_{i} = \varphi_{i+1}\circ p_i$), and $\varphi_1 = \varphi_k$. It follows that the composition of all the maps $p_i$  is the identity.


However, in most cases $\underline{L}(f)$ is not in general position, so one needs first to perturb the Lagrangians by Hamiltonian isotopies before defining its quilted Floer homology. For general Hamiltonian isotopies, it is not anymore true that the perturbed $\I(\underline{L}(f_i))$ and $\I(\underline{L}(f_{i+1}))$ are in one-to-one correspondence, nor it is true  that one can identify them with a common  moduli space $\I$ independent from the decompositions. 

Such a proof should in principle be possible to extend to the non-transverse case by using gauge theoretic perturbations such as holonomy perturbations. Unfortunately we couldn't find the right perturbations to achieve this, as standard holonomy perturbations would at best make the intersections clean along $SO(3)$-orbits, but not transverse (the intersections would be transverse at the level of character varieties, but one would need the perturbations to break the $SO(3)$-symmetry). The problem of finding such perturbations seems interesting, and would permit to conclude directly. Instead, we will prove naturality by using Cerf theory, and checking some elementary moves, in a similar spirit of \cite{JuhaszThurston}.

%
%

Let $\Fcal_0$ be the set of functions on $W$ that are Morse,  ``excellent'' (i.e. all critical values are distinct), vertical on $\partial^{vert}W$ (i.e. mapping $p(s,t)$ to $t$), and $f^{-1}(i) = \Sigma_i$, $i=0,1$ (horizontal on $\Sigma_0$  and $\Sigma_1$), and  "fiber-connected" (i.e. with all level sets connected). The space   $\Fcal_0$ corresponds to the top stratum of a natural stratification $\lbrace \Fcal_k \rbrace$, whose first strata will be defined in definition~\ref{def:stratification}. Before doing so, let us explain the strategy. 
 The space $\Fcal_0$ is disconnected, but becomes path connected when we attach to it $\Fcal_1$. 
 Now $\Fcal_0 \cup \Fcal_1$ is not simply connected, but becomes simply connected when one attach $\Fcal_2$, as we shall see in lemma~\ref{lem:Fleq2simplyconnected}. Therefore, we will mostly be interested in $\Fcal_1$ and $\Fcal_2$.
 

We view $\Fcal_0$ inside $\Fcal$, the set of all functions that are vertical on $\partial^{vert}W$, horizontal on $\Sigma_0$ and $\Sigma_1$, and fiber-connected. One can think of $\lbrace \Fcal_k \rbrace$ as  the intersection of  two stratifications  $\left\lbrace \Gcal_k \right\rbrace_k $ and  $\left\lbrace \Hcal_k\right\rbrace_k $, where $\Gcal_k$ stands for the functions with $k$ more critical points than critical values, namely:
\begin{defi} Define $\Gcal_0$, $\Gcal_1$ and $\Gcal_2$ as follows:
\begin{itemize}
\item $\Gcal_0$ consists in functions injective on their set of critical points,
\item $\Gcal_1$ consists in functions for which exactly two critical values coincide,
\item $\Gcal_2 = \Gcal_2^a \cup \Gcal_2^b $, with $\Gcal_2^a$ consisting in functions with two double critical values (i.e. $f(c_1)=f(c_2)$ and $f(c_3)=f(c_4)$, but $f(c_1)\neq f(c_3)$ for critical points $c_i$'s), and $\Gcal_2^b$ consisting in functions with one triple critical value (i.e. $f(c_1)=f(c_2)=f(c_3)$).
\end{itemize}

\end{defi}

And the first strata of $\left\lbrace \Hcal_k\right\rbrace_k $ are given by:
\begin{defi}
\begin{itemize}
\item $\Hcal_0$ consists in Morse functions,
\item $\Hcal_1$ consists in functions with Morse singularities and one $A_2$ singularity (birth-death)
\item $\Hcal_2 = \Hcal_2^a \cup \Hcal_2^b $, with $\Hcal_2^a$ consisting in functions with two  $A_2$ singularities, and $\Hcal_2^b$ functions with one  $A_3^\pm$ singularity (also known as swallowtail singularities).
\end{itemize}
\end{defi}

Let then  $\Fcal_k$ correspond to the intersection of the two stratifications:
\begin{defi}\label{def:stratification}
\begin{itemize}
\item $\Fcal_0 = \Gcal_0 \cap \Hcal_0$,
\item $\Fcal_1 = (\Gcal_0 \cap \Hcal_1) \cup (\Gcal_1 \cap \Hcal_0) $,
\item $\Fcal_2 = (\Gcal_0 \cap \Hcal_2) \cup (\Gcal_1 \cap \Hcal_1) \cup (\Gcal_2 \cap \Hcal_0)$.
\end{itemize}
\end{defi}

The following lemma mostly relies on work of Gay-Kirby \cite{GayKirby}.

\begin{lemma}\label{lem:Fleq2simplyconnected}
$\Fcal_{\leq 2} = \Fcal_0 \cup \Fcal_1 \cup \Fcal_2$ is simply connected.
\end{lemma}

\begin{proof} Without the fiber-connectedness assumption, this would follow from the fact that the complement of $\Fcal_{\leq 2}$ in $\Fcal$ has codimension $\geq 3$ (by definition of being a stratification \cite{cerf1970stratification}). Let $(f_t)_{0\leq t\leq 1}$ be a loop in $\Fcal_{\leq 2}$. We first show that one can  homotope it in $\Fcal_{\leq 2}$ to a loop $(g_t)_{0\leq t\leq 1}$ of \emph{ordered functions}, i.e. functions such that for any critical points $x$, $x'$ with Morse index satisfying $I(x)<I(x')$, one has $f(x)<f(x')$. This is a parametrized analogue of \cite[lemma~4.10]{GayKirby}.

Indeed, the critical points of the $(f_t)_{0\leq t\leq 1}$ come in families $z\colon [a,b] \to W$, with $a$ the birth-time and $b$ the death time of the critical point.

Given a pseudo-gradient $X_t$ for $f_t$, define for a family  $z\colon [a,b] \to W$ the union of the stable, respectively unstable manifolds:
\begin{align*}
S_z &= \left\lbrace (t,x)\ |\ t\in [a,b], x\in S_{z(t)}(f_t) \right\rbrace \\
U_z &= \left\lbrace (t,x)\ |\ t\in [a,b], x\in U_{z(t)}(f_t) \right\rbrace,
\end{align*}
where $S_{z(t)}(f_t)$, resp. $U_{z(t)}(f_t)$ denotes the stable, resp. unstable manifold of $z_t$ with respect to the pseudo-gradient $X_t$.

For most $t$, $U_{z(t)}(f_t)$ is a smooth submanifold of dimension $i(z)$ the Morse index of $z(t)$, except when $t$ is either $a$ or $b$, in which case it can be of dimension $i(z) +1$ or $i(z)$, depending on the index of the other critical points it dies (or born) with. In any case, $U_z$ is always of dimension at most $i(z)+1$ (i.e. a subset of a finite union of submanifolds of dimension  $i(z)+1$). Therefore, if $z$ and $z'$ are such paths, so that $i(z) < i(z')$, for a generic pseudo-gradient $X_t$, one will have $U_{z'}\cap S_z = \emptyset$.

It follows that one can reorder the critical points (see \cite[Sec.~4]{milnor_hcob}): one gets a homotopy $f_{t,s}$ so that $f_{t,0}=f_t$, and $f_{t,1}$ is ordered for each $t$. Furthermore, each $f_{t,s}$ is fiber-connected, since the $f_{t,0}$ are, and fiber-connectedness may fail when one attach an $(n-1)$-handle before a 1-handle ($n=3$ stands for $\dim W$), which doesn't happen during the reordering process. 

Now, by \cite[Th.~4.9]{GayKirby}, the loop $f_{t,1}$ can be contracted inside the space of ordered functions. It therefore follows that $f_t$ is null homotopic.
\end{proof}

It follows from lemma~\ref{lem:Fleq2simplyconnected} that the fundamental group of $\Fcal_0 \cup \Fcal_1$ is generated by \emph{meridians} of $\Fcal_2$ (i.e. loops in  $\Fcal_0 \cup \Fcal_1$ going around $\Fcal_2$). We now describe such meridian loops, after which we will show that the corresponding maps in $HSI$ are identities. There are five cases to consider, corresponding respectively to the five components of
\[
\Fcal_2 = (\Gcal_0 \cap \Hcal_2^a) \cup (\Gcal_0 \cap \Hcal_2^b) \cup (\Gcal_1 \cap \Hcal_1) \cup (\Gcal_2^a \cap \Hcal_0) \cup (\Gcal_2^b \cap \Hcal_0)
.\]

\vspace{.3cm}
\paragraph{\underline{$\Gcal_0 \cap \Hcal_2^a$}}

Let $f\in \Gcal_0 \cap \Hcal_2^a$: $f$ has two birth-death (or $A_2$) singularities $z_1$ and $z_2$, with $f(z_1)<f(z_2)$. Call $P_1 = \lbrace x_1,y_1 \rbrace$  and $ P_2 = \lbrace x_2,y_2 \rbrace$ the corresponding two pairs of critical points. A loop around $f$ can be described by a sequence of four moves:
\begin{itemize}
\item[M1] birth of $\lbrace x_1,y_1 \rbrace$,
\item[M2] birth of $\lbrace x_2,y_2 \rbrace$, 
\item[M3] death of $\lbrace x_1,y_1 \rbrace$, 
\item[M4] death of $\lbrace x_2,y_2 \rbrace$. 
\end{itemize}


Assume that for $i=1,2$, the pair $P_i$ correspond to a sequence of cobordisms $(W^{x_i},W^{y_i})\colon\Sigma^i_0\to\Sigma^i_1\to\Sigma^i_2$ that compose to a cylinder, so that $W$ has a Cerf decomposition:
\[\xymatrix{
D^2\ar[r]^{\underline{W}} & \Sigma^1_0 \ar[r]^{W^{x_1}} & \Sigma^1_1 \ar[r]^{W^{y_1}} & \Sigma^1_2 \ar[r]^{\underline{W}'} & \Sigma^2_0 \ar[r]^{W^{x_2}} & \Sigma^2_1 \ar[r]^{W^{y_2}} & \Sigma^2_2 \ar[r]^{\underline{W}''} &  D^2.
}\]

Denote $(\underline{L},L^{x_1}, L^{y_1}, \underline{L}',L^{x_2}, L^{y_2},\underline{L}'')$ the corresponding sequence in $\Symp$, so that the sequence of moves correspond to diagram~\ref{diag:twobirthdeath}, where each arrow corresponds to a geometric composition. Commutativity of diagram~\ref{diag:twobirthdeath} at the $HSI$ groups level immediately follows from  part $(A)$ of lemma~\ref{lem:commutativity_compo} applied to this sequence.

\begin{equation}\label{diag:twobirthdeath}
\xymatrix{
(\underline{L},L^{x_1}, L^{y_1}, \underline{L}',L^{x_2}, L^{y_2},\underline{L}'')  \ar[r]^{M3} & (\underline{L},L^{x_1} \circ L^{y_1}, \underline{L}',L^{x_2}, L^{y_2},\underline{L}'') \ar[d]^{M4} \\
 (\underline{L},L^{x_1}, L^{y_1}, \underline{L}',L^{x_2} \circ L^{y_2},\underline{L}'') \ar[u]^{M2} & (\underline{L},L^{x_1} \circ L^{y_1}, \underline{L}',L^{x_2} \circ L^{y_2},\underline{L}'')\ar[l]^{M1} 
}.
\end{equation}

\vspace{.3cm}
\paragraph{\underline{$\Gcal_0 \cap \Hcal_2^b$}}

We refer to  \cite[Sec.4]{JuhaszThurston} for more details and a more general discussion about  $A_3^\pm$ singularities.  An $A_3^-$-singularity is a singularity for which (in our case) the function can be modelled in a chart as:
\[
f(x_1,x_2,x_3) = -x_1^4 - x_2^2 + x_3^2 .
\]
Such a function belongs to a 2-parameter family
\[
f_{\lambda, \mu}(x_1,x_2,x_3) = -x_1^4 +\lambda x_1^2  +\mu x_1 - x_2^2 + x_3^2
.\]
In general, there can be other kinds of $A_3^-$ singularities, but since the Morse functions we are considering can only have index 1 or 2 critical points, these are the only possible behaviour. There can also be $A_3^+$, for which the local model is $f(x_1,x_2,x_3) = -x_1^4 - x_2^2 + x_3^2$, but this case can be obtained from an $A_3^-$-singularity by reversing the cobordism.

Take a circle of small enough radius on the $\lambda, \mu$-plane. It crosses $\Fcal_1$ at three points, for which either $(\lambda>0, \mu=0)$, or $8\lambda^3 -27\mu^2=0$ (two solutions), corresponding  respectively to a critical point switch, a death and a birth.

Pick two values $\lambda_0,\mu_0$ for which $\mu>0$ and $f_{\lambda_0,\mu_0}$ has three critical points $a, b, c$, of Morse index respectively 1, 2, 2, and such that $f_{\lambda_0,\mu_0}(a)<f_{\lambda_0,\mu_0}(b) <f_{\lambda_0,\mu_0}(c)$. When the point $(\lambda,\mu)$ rotates clockwise in the circle, the following moves happen:
\begin{itemize}
\item[M1] $(a,b)$ die together,
\item[M2] $c$ becomes $b$ (which we could think as a diffeomorphism equivalence),
\item[M3] $(a,c)$ born together,
\item[M4] $b$ and $c$ switch position.
\end{itemize}

\begin{remark}
One could alternatively view such a sequence as a two-function with a swallowtail singularity, as in \cite{GayKirby}.
\end{remark}

Let $\Sigma_0,\Sigma_1,\Sigma_2^c,\Sigma_3$ be level surfaces of $f_{\lambda_0,\mu_0}$, that are part of a Cerf decomposition: so that $a$ is between $\Sigma_0$ and $\Sigma_1$,  $b$ is between $\Sigma_1$ and $\Sigma_2^c$, and  $c$ is between $\Sigma_2^c$ and $\Sigma_3$. Let also  $\Sigma_2^b$ be a level surface that separates $b$ and $c$ after they switch position. The situation is summarized in figure~\ref{fig:birth_death_birth}.

\begin{figure}[!h]
    \centering
    \def\svgwidth{.60\textwidth}
    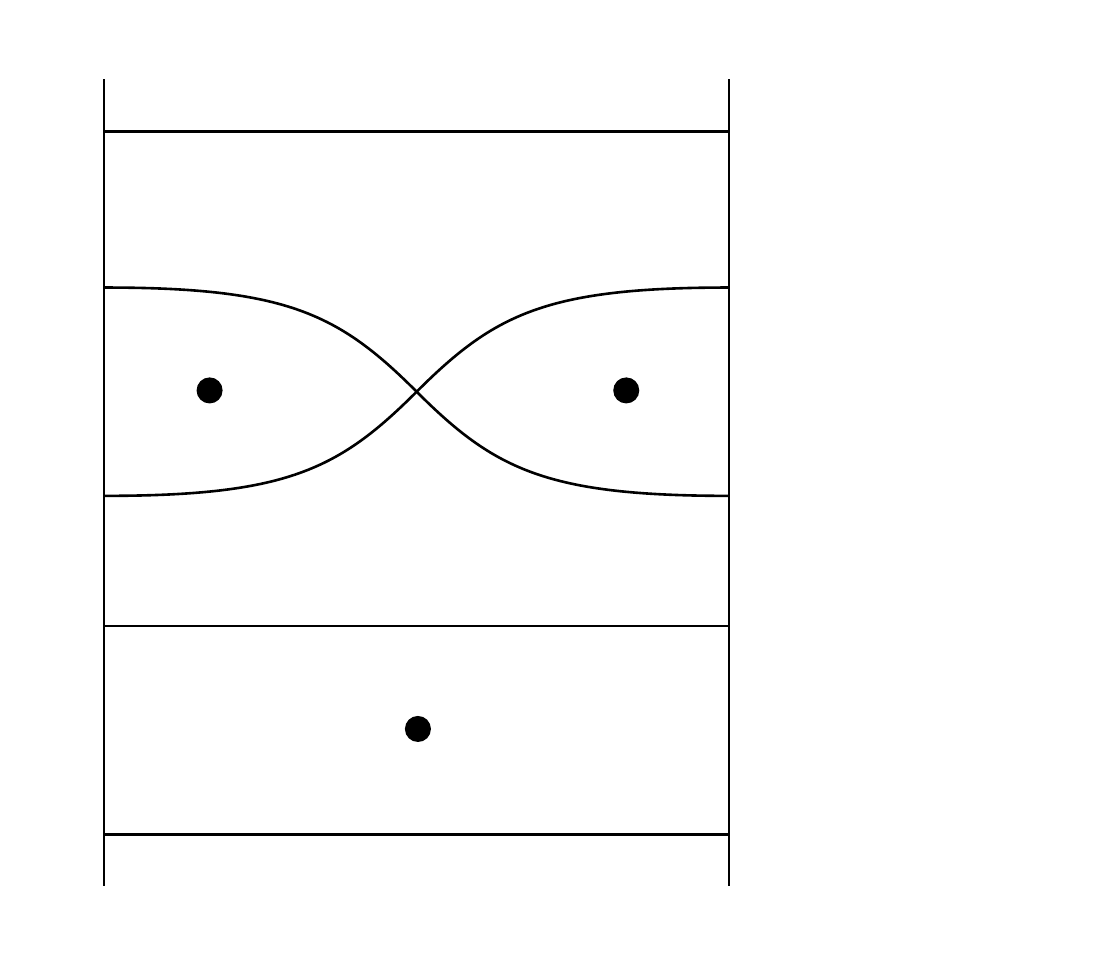
      \caption{}
      \label{fig:birth_death_birth}
\end{figure} 

The sequence of moves can be summarized in the central rectangle of diagram~\ref{diag:birth_death_birth}, where each arrow between two surfaces represents the part of $W$ between these surfaces. Notice that the moves $M2$ and $M4$ can be factored: for $M2$ we used the fact that the cobordisms $\Sigma_0 \to \Sigma_2^b$ and $\Sigma_0 \to \Sigma_2^c$, containing respectively the cancelling pairs $(a,c)$ and $(a,b)$, are cylinders. For $M4$ we composed the two $2$-handle attachment $\Sigma_1 \to \Sigma_3$. This is not an elementary cobordism anymore, but a compression body. What matters is that after applying the FFT functor, the Lagrangian correspondences still compose in an embedded$^+$ way (see \cite[Th.~3.22]{surgery}). After applying the Floer Field Theory (FFT) functor we get a similar diagram  in $\Symp$. Now, observe that the two squares in dashed arrows are of the form of the one in lemma~\ref{lem:commutativity_compo}, part (A), and therefore induce commutative diagrams after taking quilted Floer homology (and completing the diagrams with the \glag\ $\underline{L}$ and $\underline{L}'$ corresponding to the rest of $W$). Commutativity of the central rectangle follows.
 
 \begin{equation}\label{diag:birth_death_birth}
 \xymatrix{
 &  \left(\Sigma_0 \to \Sigma_1 \to \Sigma_3 \right) \ar@{-->}[rd]\ar@{-->}[ddd]& \\ 
 \left( \Sigma_0 \to \Sigma_1  \to \Sigma_2^b  \to \Sigma_3 \right) \ar@{-->}[ru]\ar[rr]^{M4}& & \left( \Sigma_0 \to \Sigma_1  \to \Sigma_2^c  \to \Sigma_3 \right) \ar[d]^{M1}\\ 
 \left( \Sigma_0 \to \Sigma_2^b \to \Sigma_3 \right) \ar[u]^{M3}& &  \left(\Sigma_0 \to \Sigma_2^c \to \Sigma_3 \right)\ar@{-->}[ld]\ar[ll]^{M2}\\ 
 &  \left(\Sigma_0 \to \Sigma_3\right)\ar@{-->}[lu]& . 
 }\end{equation}

\vspace{.3cm}
\paragraph{\underline{$\Gcal_1 \cap \Hcal_1$}} $f$ has one birth-death singularity, of a pair of critical points $(a,b)$ say, and two critical points $c,d$ for which $f(c)=f(d)$. A loop around this singularity can be described by the following moves:
\begin{itemize}
\item[M1] birth of $(a,b)$,
\item[M2] $c$ and $d$ switch,
\item[M3] death of $(a,b)$,
\item[M4] $d$ and $c$ switch.
\end{itemize}
Let $W^a\colon \Sigma_0 \to \Sigma_1$ and $W^b\colon \Sigma_1 \to \Sigma_2$ stand respectively for the handle attachments of $a$ and $b$, likewise let $W^c\colon \Sigma_3 \to \Sigma_4$, $W^d\colon \Sigma_4 \to \Sigma_5$ stand respectively for the handle attachments of $c$ and $d$, when $c$ is below $d$, and $W^{d'}\colon \Sigma_3 \to \Sigma_4'$, $W^{c'}\colon \Sigma_4' \to \Sigma_5$,  stand respectively for the handle attachments of $c$ and $d$, when $c$ is above $d$, so that $W^c\cup_{\Sigma_4}W^d = W^{d'}\cup_{\Sigma_4'}W^{c'}$.

Assume that the pair $(a,b)$ comes before $(c,d)$ in $W$ (the other case is analogous) and denote $\underline{V}$,  $\underline{V}'$ and  $\underline{V}''$, the complements of these handle attachments in $W$, equipped with fixed Cerf decompositions, so that for example:
\[
W = \xymatrix{D^2 \ar[r]^{\underline{V}} & \Sigma_0 \ar[r]^{W^a} &  \Sigma_1 \ar[r]^{W^b} & \Sigma_2  \ar[r]^{\underline{V}'} &  \Sigma_3 \ar[r]^{W^c} & \Sigma_4  \ar[r]^{W^d} & \Sigma_5  \ar[r]^{\underline{V}''} & D^2 }.
\]

Denote $\underline{L}$, $\underline{L}'$, $\underline{L}''$, $L^a$, $L^b$, $L^c$, $L^d$, $L^{d'}$, $L^{c'}$ the corresponding (generalized) Lagrangian correspondences. The sequence of moves can now be described in diagram~\ref{diag:birthdeath_switch} below. In dashed lines we have decomposed the critical point switch moves, by composing the correspondences $L^c \circ L^d =L^{d'}\circ L^{c'}$, which compose in an embedded$^+$ way.
 \begin{equation}\label{diag:birthdeath_switch}
 \xymatrix{
 &  (\underline{L},L^a, L^b, \underline{L}',L^c\circ L^d, \underline{L}'') \ar@{-->}[rd]\ar@{-->}[ddd]& \\ 
 (\underline{L},L^a, L^b, \underline{L}',L^c, L^d, \underline{L}'') \ar@{-->}[ru]\ar[rr]^{M2}& &(\underline{L},L^a, L^b, \underline{L}',L^{d'}, L^{c'}, \underline{L}'') \ar[d]^{M3}\\ 
 (\underline{L},L^a \circ L^b, \underline{L}',L^c, L^d, \underline{L}'') \ar[u]^{M1}& &  (\underline{L},L^a \circ L^b, \underline{L}',L^{d'}, L^{c'}, \underline{L}'')\ar@{-->}[ld]\ar[ll]^{M4}\\ 
 &  (\underline{L},L^a \circ L^b, \underline{L}',L^c \circ L^d, \underline{L}'')\ar@{-->}[lu]& . 
 }\end{equation}
Commutativity at the quilted Floer homology level again follows from lemma~\ref{lem:commutativity_compo} applied to the two dashed squares.
%

\vspace{.3cm}
\paragraph{\underline{$\Gcal_2^a \cap \Hcal_0$}} $f$ has two  double critical values (i.e. $f(a)=f(b)$ and $f(c)=f(d)$) 

This is again a consequence of lemma~\ref{lem:commutativity_compo}, part $(B)$. One would get a square, where each arrow of the square correspond to a critical point switch, and can be factored. In the end one can divide the square to four squares as in diagram~\ref{diag:two_double_critval} below, where each dashed arrow corresponds to a composition. 
In diagrams~\ref{diag:two_double_critval} and \ref{diag:triple_crit_value}, we omitted to label some vertices for sake of clarity. For example, the middle vertex of the first row in diagram~\ref{diag:two_double_critval} corresponds to $f(a)=f(b); f(c) < f(d)$.

 \begin{equation}\label{diag:two_double_critval}
 \xymatrix{
f(a)<f(b); f(c) < f(d) \ar@/^1pc/[rr]^{a \leftrightarrow b} \ar@{-->}[r] \ar@{-->}[d] &  \ar@{-->}[d] &  f(a)>f(b); f(c) < f(d) \ar@/^1pc/[dd]^{c \leftrightarrow d} \ar@{-->}[l] \ar@{-->}[d] \\ 
\ar@{-->}[r] &  & \ar@{-->}[l]  \\ 
 f(a)<f(b); f(c) > f(d)\ar@/^1pc/[uu]^{c \leftrightarrow d} \ar@{-->}[r] \ar@{-->}[u]& \ar@{-->}[u] & f(a)>f(b); f(c) > f(d) \ar@/^1pc/[ll]^{a \leftrightarrow b}  \ar@{-->}[l] \ar@{-->}[u] .
 }\end{equation}

\vspace{.3cm}
\paragraph{\underline{$\Gcal_2^b \cap \Hcal_0$}}
Let $f_0$ be a function so that $f_0(a)=f_0(b)=f_0(c)$, for three distinct critical points. A nearby function $f$ can have six different behaviours, depending on the relative positions of $f(a)$, $f(b)$, and $f(c)$, and these are related by critical point switches, corresponding to the arrows in diagram~\ref{diag:triple_crit_value}. 
We aim to show it induces a commutative diagram at the quilted Floer homology level. First notice that after applying the FFT functor, each possible compositions are embedded$^+$, therefore each arrow in the diagram can be factored through an intermediate \glag, where one has composed the two correspondences involved in the corresponding twist. Now each of the six new \glag s can also be further composed to the Lagrangian correspondence associate with the three handle attachments. Therefore the hexagon of diagram~\ref{diag:triple_crit_value} can be split into six 4-gons, which by lemma~\ref{lem:commutativity_compo} part $(A)$ are commutative after taking quilted Floer homology. That implies that diagram~\ref{diag:triple_crit_value} commutes.
 \begin{equation}\label{diag:triple_crit_value}
 \xymatrix{  
 &   & f(b)<f(a)<f(c) \ar@/^1pc/[rrdd]^{a \leftrightarrow c} \ar@{-->}[rd]\ar@{-->}[ld] &  &  \\
 &  \ar@{-->}[rdd] &   &  \ar@{-->}[ldd] &  \\
f(a)<f(b)<f(c) \ar@/^1pc/[rruu]^{a \leftrightarrow b} \ar@{-->}[ru] \ar@{-->}[d] &   &   &  & f(b)<f(c)<f(a) \ar@/^1pc/[dd]^{b \leftrightarrow c} \ar@{-->}[d] \ar@{-->}[lu]  \ \\
\ar@{-->}[rr]&   &   &  & \ar@{-->}[ll]\\
f(a)<f(c)<f(b) \ar@/^1pc/[uu]^{b \leftrightarrow c}\ar@{-->}[u] \ar@{-->}[rd]  &   &   &  & f(c)<f(b)<f(a)\ar@/^1pc/[lldd]^{a \leftrightarrow b} \ar@{-->}[ld] \ar@{-->}[u]  \\
 &  \ar@{-->}[ruu] &   & \ar@{-->}[luu] &  \\
 &   & f(c)<f(a)<f(b)\ar@/^1pc/[lluu]^{a \leftrightarrow c} \ar@{-->}[ru] \ar@{-->}[lu]  &  &   .
}\end{equation}


\subsection{Mapping class group and fundamental group representations}\label{ssec:mcgrep}

Assume now we are given a diffeomorphism $\varphi\colon Y \to Y'$ mapping $z$ to $z'$. Let $W$ be the blow up of $Y$ at $z$, and pick a Cerf decomposition $f$ of it. Then $\varphi$ induces a Cerf decomposition of $Y'$,  call it $\varphi_*f$. To $(Y, \varphi^*P, z,f)$ is associated a \glag\ $\underline{L}$, and to $(Y', P, z',\varphi_*f)$ is associated another \glag\ $\underline{L}'$. 

The map $\varphi$ gives an identification of $\underline{L}$ and $\underline{L}'$: it furnishes a symplectomorphism between each symplectic manifolds appearing in the sequences, and these symplectomorphisms map the first Lagrangians to the second. This identification therefore provides an isomorphism $HF(\underline{L}) \to HF(\underline{L}')$, which we define to be  the map
\[
F_{\varphi}\colon HSI(Y,\varphi^* P,z) \to HSI(Y,P,z).
\]

In particular, when $P$ is trivial, one gets a representation of the mapping class group of $(Y,z)$ on $HSI(Y,z)$. More concretely, If a Cerf decomposition (or a Heegaard splitting) is fixed, giving a \glag\ $\underline{L}$, to get automorphisms of $HF(\underline{L})$ one then needs to relate the new decomposition $\varphi_*f$ with $f$ by a sequence of Cerf moves.

%

We now define  the action of the fundamental group.  Let $\gamma\colon [0,1]\to Y$ be a smooth embedded path based at $z$. Pick a vector field $X$ on $Y$ extending $\gamma
'(t)$ and being zero outside a tubular neighborhood of $\gamma$, and let $\varphi$ be its time one flow. Take then $F_\varphi\colon HSI(Y,\varphi^* P,z) \to HSI(Y,P,z)$ to be the map associated to $\gamma$. Moreover, there is a preferred homotopy between $\varphi^* P$ and $P$ given by $\varphi_t^* P$, with $\varphi_t$ the time $t$ flow of $X$, for $0\leq t\leq 1$. This homotopy identifies $ HSI(Y,\varphi^* P,z)$ with $HSI(Y,P,z)$ Composing  $F_\varphi$ with this identification, one then gets an automorphism of $HSI(Y,P,z)$. This automorphism is an isotopy invariant of $\gamma$, but it is not clear a priori that it is a homotopy invariant. This is also true, and will follow from the construction in section~\ref{sec:cob}, and the following observation: to a path  $\gamma\colon [0,1]\to Y$ corresponds the path $\tilde{\gamma} \colon t\mapsto (\gamma(t),t)$ in $Y\times [0,1]$, and a homotopy of $\gamma$ yields an isotopy of $\tilde{\gamma}$.

%

\subsection{Action of $H^1(Y; \Z{2})$ for Lens spaces}\label{ssec:action_H1}

Let $Y = L(p,q)$ be a lens space, with $p=2n$ even, and $P=SU(2)\times Y$ the trivial bundle. Let $\widehat{\Sigma}$ be a genus one Heegaard splitting, $z\in \widehat{\Sigma}$ a base point, and $\Sigma$ the real oriented blow-up. One can identify $\N(\Sigma, P_{\Sigma})$ with the set $\left\lbrace (A,B)| [A,B]\neq -1 \right\rbrace $, so that $L_0 = \left\lbrace (A,B)|B =1 \right\rbrace $ and  $L_1 = \left\lbrace (A,B)|A^p B^{-q} =1 \right\rbrace $. \MW\ showed \cite[Prop.~7.3]{MW} that $HSI(L(p,q)) = H_*(L_0 \cap L_1;\Z{2})$, and  $L_0 \cap L_1$, seen as a subset of $L_0\simeq SU(2)$, consists in the union of conjugacy classes $S_a$ of $e^{\frac{ia\pi}{p}}$, for $a= 0, \cdots n$. These are points when $a=0,n$, and two-spheres when $a= 1, \cdots n-1$.

$H^1(Y;\Z{2})\simeq \Z{2}$ acts on $\N(\Sigma, P_{\Sigma})$ by the involution $(A,B) \mapsto (-A,B)$. It acts on $L_0 \cap L_1$ by switching $S_a$ and $S_{n-a}$, and the action on $HSI$ corresponds to the induced action on homology.

\begin{remark}The groups $HSI(Y)$ and $\widehat{HF}(Y)$ have the same rank for all known examples, \MW\ asked whether this is always true. One could ask whether a similar action of $H^1(Y;\Z{2})\simeq \Z{2}$ exists on $\widehat{HF}(Y)$, and whether it is related to involutive Heegaard Floer homology, as defined by Hendricks and Manolescu \cite{HendricksManolescuInvolutive}.
\end{remark}

\section{Maps from cobordisms}\label{sec:cob}

From now on $W$ will no longer stand for a 3-dimensional cobordism as it used to stand in the last section, but rather for a 4-dimensional cobordism.

\subsection{General framework}
\label{ssec:bigpic}

So far we can associate  a map $F_\varphi$ to a  diffeomorphism $\varphi$ fixing the basepoint. We will see that $HSI$ carries a structure close to a (3+1)-TQFT, that is for each four-dimensional cobordism $W\colon Y \to Y'$ is associated a linear map between the groups associated to its boundary, and these maps satisfy nice gluing properties. The picture will be slightly different however: due to the presence of basepoints and $SO(3)$-bundles, the maps will be associated to cobordisms with arcs connecting these points, and $SO(3)$-bundles. 

Moreover, $HSI$ can also be constructed by defining a functor $Cob_{2+1} \to \Symp$, i.e. a $(2+1)$-Floer field theory (abbreviated FFT), see \cite[sec.~3]{surgery} and then applying quilted Floer homology. These two functorial  points of view shoud be brought together: in principle $HSI$ and the cobordism maps should come from a (slight variation of a) (2+1+1)-\fft: the category $\Symp$ should be endowed with a structure close to a 2-category, using quilted Floer homology as 2-morphism spaces, see \cite{WWfft}.  
We plan in future work \cite{ham} to extend such a functorial picture down to dimension 1.

Nevertheless we find instructive to have this point of view in mind when constructing the maps $F_W$, although our construction will eventually be analogous to \cite{OSholotri}. According to this principle, it suffices to assign a Floer homology class for elementary cells (namely the cells for 1-handles, 2-handles and 3-handles), seen as manifolds with corners (i.e. 2-morphisms in $Cob_{2+1+1}$), and then use the operations of the 2-category $\Symp$ (i.e. composition of 2-morphisms, using quilted pair of pants). We now explain how to choose these Floer homology classes. Let $X$ be a  4-manifold with corners with $\partial X = Y_1 \cup_{\Sigma} Y_2$ ($\Sigma$ is the codimension 2 corner). After equipping $X$ with a metric with cylindrical ends, and a principal bundle over it, one should be able to define a moduli space $\M_{ASD}(X)$ of anti-self dual instantons (with suitable behaviour near the basepoint). Fixing a temporal gauge and taking the limit near the boundary should give two restriction maps to the moduli spaces of flat connexions of $Y_1$ and $Y_2$, which also restricts to the extended moduli space of $\Sigma'$ ($\Sigma' = \Sigma\setminus disk$), as in the following diagram.

\[\xymatrix{ & \M(Y_1) \ar[rd] &   \\  \M_{ASD}(X) \ar[ru] \ar[rd]    &  & \N(\Sigma '). \\  & \M(Y_2) \ar[ru] &  }\] 

One then obtains a map $r\colon \M_{ASD}(X) \to L(Y_1) \cap L(Y_2)\subset \N(\Sigma ')$. Then the fundamental class (say it exists) $ \left[ \M_{ASD}(X)\right] $ pushes forward to a class  $ r_* \left[ \M_{ASD}(X)\right] $ in $H_*(L(Y_1) \cap L(Y_2))$, which (say the intersection is clean) is the first page of a spectral sequence of $E_\infty$-page $HF(L(Y_1), L(Y_2))$. If furthermore $ r_* \left[ \M_{ASD}(X)\right] $ descends to a class in $HF(L(Y_1), L(Y_2))$, then we might use this class.  It is likely that the lowest dimensional part of $\M_{ASD}(X)$ corresponds to the classes involved in the construction of the maps $F_{W,P}$ (the classes $C_+$ and $C$), and it should be possible to use its higher dimensional components in order to define $HSI$-valued Donaldson polynomials.

\subsection{Construction}
Recall the setting:  $W$ is a compact connected oriented smooth 4-cobordism from $Y$ to $Y'$ (both closed connected oriented 3-manifolds), $P$ an $SO(3)$-bundle over $W$, and $\gamma\colon [0,1]\to W$ a path from  $Y$ to $Y'$ connecting two basepoints $z$ and $z'$.

 
We proceed analogously to  \cite{OSholotri}: we first cut $W$ in elementary  cobordisms, corresponding to single handle  attachments, then define morphisms associated to such cobordisms, and finally check that the morphism obtained by composing those doesn't  depend on the decomposition.

\subsubsection{1-handle and  3-handle attachment}\label{ssec:1handle3handle}

Let  $W$ be a 4-cobordism between $Y$ and $Y'$ corresponding to a 1-handle  attachment  to $Y$. The manifold  $Y'$ is homeomorphic to the connected sum $(S^2\times S^1)\# Y$, and $W$ is homeomorphic to the  boundary connected sum   
\[
 W \simeq (D^3 \times S^1)\#_\partial Y \times [0,1].
 \] 
 
 Take an embedded $S^2\times [0,1] \subset W$ that separates $W$ into two pieces, one diffeomorphic to $(Y\setminus D^3) \times [0,1]$, the other to $D^3 \times S^1$. Assume that the base path of $W$ is entirely contained in $S^2\times [0,1]$ and is of the form $\gamma(t) = (z,t)$, for some $z\in S^2$.
Take a Cerf decomposition $\Sigma_2, \cdots$ of $Y$ (after blowing up at $z$). It induces a Cerf decomposition $\Sigma_0', \Sigma_1', \Sigma_2', \cdots$ of $Y'$, with $\Sigma_1'$ a (punctured) genus one Heegaard splitting of $S^2\times S^1$, $\Sigma_0$ and $\Sigma_2$ discs, and $\Sigma_k$ corresponding to $\Sigma_k'$.

Let $P$ be an $SO(3)$-bundle over $W$. Denote $\underline{L} = (L_{23},L_{34},\cdots )$ and $\underline{L}' = (L_{01},L_{12},L_{23},\cdots )$ the \glag\ associated respectively with $(Y, P_{|Y})$ and $(Y', P_{|Y'})$, with the given Cerf decompositions.

The restrictions of $P$ to $(Y\setminus D^3) \times \lbrace 0 \rbrace$ and  $(Y\setminus D^3) \times \lbrace 1 \rbrace$ are homotopic through the path of pullbacks of $P_{|(Y\setminus D^3) \times \lbrace t \rbrace}$, so one can identify $(L_{23},L_{34},\cdots )$ and $(L_{23}',L_{34}',\cdots )$ canonically. Moreover the restriction of $P$ to $(S^2\times S^1)\setminus B^3 \subset Y'$ is trivial, since it extends to $D^3 \times S^1$ (which has $H^2(D^3 \times S^1;\Z{2}) = 0$). It follows that, seen as subsets of $\N(\Sigma_1)$, we have $L_{01} = L_{12}$ (both $\N(\Sigma_0)$ and $\N(\Sigma_2)$ are points). By the Künneth formula \cite[Th~1.1]{surgery}, one has
\[
HSI(Y',P') = HSI(S^2\times S^1) \otimes HSI(Y,P)  .
\]
\begin{remark} From the genus one Heegaard splitting of $S^2\times S^1$, one can see that the action of $H^1(S^2\times S^1;\Z{2})$ on $HSI(S^2\times S^1)$ is trivial, therefore the previous formula is not ambiguous with respect to this action.
\end{remark}



In order to be able to refer to its classes, we fix the following absolute grading: $HSI(S^2\times S^1) = \Z{2}^{(3)} \oplus \Z{2}^{(0)}$, where the exponent stands for  the degree modulo 8. Under this identification, denote by $C_+\in \Z{2}^{(3)}$ the  non-trivial element of $\Z{2}^{(3)}$.


\begin{defi} 
\begin{itemize}
\item (Map associated to a 1-handle attachment). We define the map $F_{W, P,\gamma} (x)$ by:
\[F_{W, P,\gamma} (x) = C_+ \otimes x.\] 
\item (Map associated to a 3-handle attachment). The cobordism $\overline{W}$ endowed with the opposite orientation, seen as a  cobordism from $Y'$ to $Y$, corresponds to a 3-handle attachment. Let also $\overline{\gamma}(t) = \gamma(1-t)$. We similarly define $F_{\overline{W}, P, \overline{\gamma}} (C_+ \otimes x) =  x$, and $F_{\overline{W}, P, \overline{\gamma}} (C_- \otimes x) = 0$, if $C_-$ stands for  the   degree 0 generator  of $HSI(S^2\times S^1)$.
\end{itemize}

\end{defi}

\begin{remark} We will see that the definition of these handle maps are constrained by the  handle creation/cancellation invariance with the 2-handles.
\end{remark}

\subsubsection{2-handle attachment}

Let $Y$ be a 3-manifold, $K\subset Y$ a framed knot,  $W$ the cobordism corresponding  to the 2-handle attachment along $K$, and $P$ an $SO(3)$-bundle over $W$. Let $Y'$ be the second  boundary of $W$, which corresponds  to the zero surgery along $K$. Denote $T = \partial (Y\setminus \nu K)$ the torus bounding a fixed tubular neighborhood $\nu K$ of $K$.


Fix the basepoint $z$ of $Y$ in the torus $T$, let $z'\in Y'$ correspond to the same point, and $\gamma(t) = (z,t)$. Denote $\lambda, \mu \subset T$ a longitude and a meridian of $K$ avoiding the  point $z$, $T'$ the torus $T$ blown up at $z$, $\underline{L} = \underline{L}(Y\setminus K , c )$ the \glag\ from  $\Nc(T')$ to $pt$. Let also $L_0, L_1 \subset \Nc(T')$ correspond to the Dehn fillings along $\mu$ and $\lambda$ respectively, so that 
\begin{align*}
HSI(Y, P_{Y},z)  =&  HF(L_0, \underline{L}), \\
HSI(Y', P_{Y'},z')  =&  HF(L_1, \underline{L}).
\end{align*}

Since the union of the two solid torus corresponding to $L_0$ and $L_1$ is a 3-sphere, $L_0$ and $L_1$ intersect transversely at one point, and one has $HF(L_1,L_0) = HSI(S^3) = \Z{2}$. 
%
%
Let $C$ be the generator corresponding to the intersection point. 

Let  $\Phi\colon HF(L_1, L_0) \otimes HF(L_0, \underline{L}) \to HF(L_1, \underline{L})$ be the quilted pair-of-pant product, defined by  counting quilted pair-of-pants as in figure~\ref{fig:pairofpants}.

\begin{figure}[!h]
    \centering
    \def\svgwidth{.55\textwidth}
    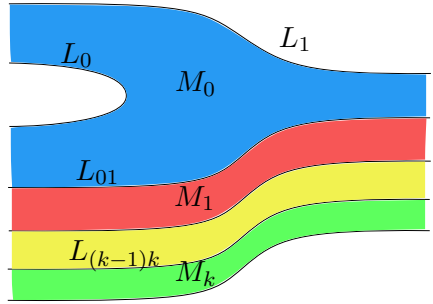
      \caption{The pair-of-pants map.}
      \label{fig:pairofpants}
\end{figure}

\begin{defi}(Map associated to a 2-handle attachment). We define $F_{W,P,\gamma}$ by 
\[
F_{W,P,\gamma} (x) = \Phi(C \otimes x).
\]
\end{defi}

\subsection{Independence of the decomposition, naturality}

Suppose that we have decomposed the cobordism $(W,\gamma)$ in $k$ elementary ones:
\[ (W,\gamma) = (W_1,\gamma_1) \cup_{Y_1, z_1} (W_2,\gamma_2) \cup_{Y_2, z_2} \cdots \cup_{Y_{k-1}, z_{k-1}} (W_k,\gamma_k),\]
with $Y_0 = Y$, $Y_k = Y'$, and $W_i$ a cobordism from $Y_{i-1}$ to $Y_i$ corresponding to the attachment of a 1,2 or 3-handle. Assume also that the path $\gamma$ is decomposed in paths $\gamma_i$ that are "horizontal" as in the previous section. 

Denote $P_i$ and $P_{|Y_i}$ the restrictions of $P$ to $W_i$ and $Y_i$ respectively. 
For each elementary piece  $(W_i,P_{i}, \gamma_i)$, we  defined a morphism 
\[
F_{W_i,P_i, \gamma_i} \colon HSI(Y_{i-1},P_{|Y_{i-1}},z_{i-1}) \to HSI(Y_i,P_{|Y_i},z_i). 
\]
 We then define $F_{W,P,\gamma}\colon HSI(Y,P_{|Y},z) \to HSI(Y',P_{|Y'},z')$ as the composition:
\[F_{W,P,\gamma} =F_{W_k,P_k,\gamma_k}\circ \cdots \circ   F_{W_2,P_2,\gamma_2}   \circ F_{W_1,P_1,\gamma_1}.\]

For this construction to make sense we first need to show that the handle maps $F_{W_i,P_i,\gamma_i}$ commute with the $Y$-maps, so that they really define the same map from $HSI(Y_{i-1},P_{|Y_{i-1}}, z_{i-1})$ to $HSI(Y_i,P_{|Y_{i}}, z_i)$ regardless of the Cerf decompositions of $Y_{i-1}$ and $Y_i$. We do so in section~\ref{ssec:commYmaps}.

We then prove that the map $F_{W, P, \gamma}$ doesn't depend on the Cerf decomposition of $W$. According to Cerf theory (see for example  \cite{WWfft}), it suffices to check that the morphisms remain unchanged after a handle creation/cancellation, and a critical point switch, as the moves corresponding to a diffeo-equivalence and a  trivial cobordism attachment are  clearly satisfied.

\subsubsection{Commutativity with the $Y$-maps} \label{ssec:commYmaps}

Let $W$ be a 4-cobordism corresponding to a handle attachment from $Y$ to $Y'$. After picking two Cerf decompositions of $Y$ and $Y'$ that are compatible with the handle attachment as in the previous construction, one gets two \glag s  $\underline{L}$ and  $\underline{L}'$, and a map 
\[F_{W,P}\colon HF(\underline{L}) \to HF(\underline{L}').\] 
If one had chosen different Cerf decompositions for $Y$ to $Y'$, one would have obtained two other \glag s $\widetilde{\underline{L}}$ and  $\widetilde{\underline{L}}'$, and another map 
\[
\widetilde{F}_{W,P}\colon HF(\widetilde{\underline{L}}) \to HF(\widetilde{\underline{L}}').
\]
 One also has two isomorphisms
 \begin{align*}
 \Theta &\colon HF(\underline{L}) \to HF(\widetilde{\underline{L}})\text{, and} \\
 \Theta' &\colon HF(\underline{L}') \to HF(\widetilde{\underline{L}}').
 \end{align*}

   We will show that the following diagram commutes:
\[ \xymatrix{ HF(\underline{L}) \ar[r]^{F_{W,P}} \ar[d]^{\Theta} & HF(\underline{L}')\ar[d]^{\Theta'}  \\ 
HF(\widetilde{\underline{L}}) \ar[r]^{\widetilde{F}_{W,P}} & HF(\widetilde{\underline{L}}') .}\]

Without loss of generality we will assume that $\Theta$ is a single $Y$-map.

\vspace{.3cm}
\paragraph{\underline{1-handles and 3-handles}} Let $(W, P, \gamma)$ be an $SO(3)$ 4-cobordism with a basepath, corresponding to a 1-handle attachment to $(Y,P_{|Y},z)$ as in section~\ref{ssec:1handle3handle}. 

Let $\underline{L}$ be a \glag\ obtained from some Cerf decomposition of $(Y,P_{|Y},z)$. Then $\underline{L}' = (L_0, (L_0)^T, \underline{L})$ corresponds to $(Y',P_{|Y'},z')$, and \[F_{W, P, \gamma} ([x]) = [C_+ \otimes x ]\text{, for }[x] \in HF(\underline{L}).\] 
Assume that $\widetilde{\underline{L}}$ is another \glag\ corresponding to another Cerf decomposition of $(Y,P_{|Y},z)$, that differs from  a Cerf move. $\underline{L}$ and $\widetilde{\underline{L}}$ differ by an embedded composition, and there is a $Y$-map  \[\Theta\colon HF(\underline{L}) \to HF(\widetilde{\underline{L}}).\]

Likewise, $\widetilde{\underline{L}}' = (L_0, (L_0)^T, \widetilde{\underline{L}})$ is the corresponding \glag\ for $(Y',P_{|Y'},z')$, and the $Y$-map  \[\Theta'\colon HF(\underline{L}') \to HF(\widetilde{\underline{L}}').\] 

One can "unsew" the second patch appearing in  $\Theta'$, by which we mean the following: if $(\underline{S}, \underline{M}, \underline{L})$ is a quilt, such that for some patch $P$ in $\underline{S}$, $M_P$ is a point, then "unsewing" corresponds to removing the patch $P$ in $\underline{S}$, and taking the same decorations. This operation doesn't change the moduli spaces and associated invariants.

Consequently, one has \[\Theta' = Id_{HF(L_0, L_0)}\otimes \Theta.\]
From the fact that $\widetilde{F}_{W,P,\gamma}\colon HF(\widetilde{\underline{L}}) \to HF(\widetilde{\underline{L}}')$ is defined by 
\[
\widetilde{F}_{W,P,\gamma}([x]) = [C_+ \otimes x],
\]
 it follows that $\widetilde{F}_{W,P,\gamma}\circ \Theta  =\Theta' \circ {F}_{W,P,\gamma} $, as announced. The claim for 3-handles follows from the same observation.

\vspace{.3cm}
\paragraph{\underline{2-handles}} Assume now that both $F_{W, P, \gamma}$ and $\widetilde{F}_{W, P, \gamma}$ correspond to a 2-handle attachment. By standard gluing arguments, the maps  $\widetilde{F}_{W,P,\gamma}\circ \Theta $ and $\Theta' \circ {F}_{W,P,\gamma} $ correspond to the contraction with $C\in HF(L_0,L_1)$ of the maps defined by counting quilts as in figure~\ref{fig:Ymaps2handles}. The equality  $\widetilde{F}_{W,P,\gamma}\circ \Theta  =\Theta' \circ {F}_{W,P,\gamma} $ then follows from the obvious homotopy relating the two surfaces.

\begin{figure}[!h]
    \centering
    \def\svgwidth{.65\textwidth}
    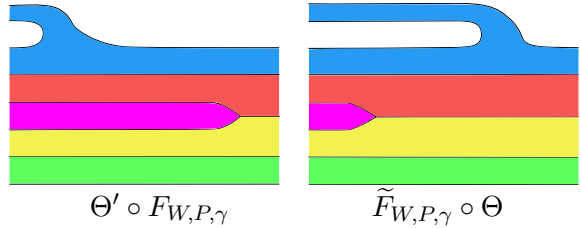
      \caption{quilted surfaces inducing the maps $\widetilde{F}_{W,P,\gamma}\circ \Theta $ and $\Theta' \circ {F}_{W,P,\gamma}$.}
      \label{fig:Ymaps2handles}
\end{figure}

\subsubsection{Birth/death} We will show that for the consecutive attachment  of a cancelling 1-handle and a 2-handle, the composition of the induced maps is the identity. The case of a 2-handle and a 3-handle can be obtained from this by  reversing the cobordism.

The situation we will describe is summarized in figure \ref{fig:naissancemort}: 

\begin{figure}[!h]
    \centering
    \def\svgwidth{.65\textwidth}
    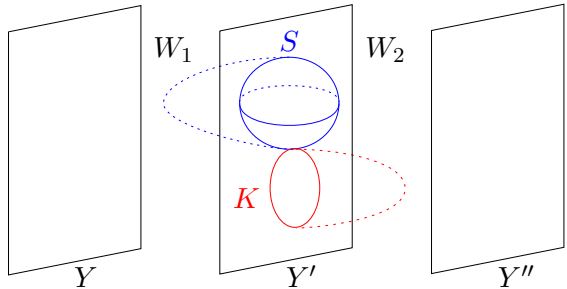
      \caption{Birth/Death.}
      \label{fig:naissancemort}
\end{figure}

Let $Y'$ be a 3-manifold, $S\subset Y'$ a 2-sphere and $K\subset Y'$ a framed knot, such that $K$ and $S$ intersect transversely at a single point $p$. Let $W_1$ denote the opposite of the cobordism corresponding to a 3-handle attachment along $S$. Denote $Y$ the other boundary of $W_1$. Let $W_2$ be the cobordism corresponding to the attachment of a 2-handle along $K$, with respect to the given framing. Let $Y''$ be the other boundary of $W_2$. Let $N$ be a regular  neighborhood of $S\cup K$ in $Y'$, $N\simeq (S^2\times S^1)\setminus B^3$, with  $B^3$ a 3-ball. The boundary $\partial N$ is a 2-sphere, $Y\simeq Y''\simeq (Y'\setminus N)\cup B^3$,  and $Y'\simeq Y\# (S^2\times S^1)$.

Let $C$ be a circle in $S$ disjoint from  $p = K\cap S$, and $T\subset N$ the torus corresponding to $C\times K$ under the identification $N\simeq (S\times K)\setminus B^3$. The torus $T$ separates $N$ in respectively a neighborhood $\nu K$ of $K$, and $N\setminus \nu K$. We will still denote $T$ the  corresponding torus in $Y''$.

Let $P$ be an $SO(3)$-bundle over $W$, denote $P_1$, $P_2$, $P_{|Y}$,  $P_{|Y'}$,  $P_{|Y''}$, its restrictions to ${W_1}$, ${W_2}$, ${Y}$, $Y''$ and ${Y'}$ respectively.



Choose basepoints $z,z',z''$ corresponding to a same point in $T$, and horizontal arcs $\gamma_1$, $\gamma_2$ connecting them. Let  $\lambda,\mu\subset T$ be a longitue and a meridian of $K$ avoiding $z$, and $T'$ the torus $T$ blown up at $z$. Denote, in accordance with the previous  paragraph, $L_0$ and $L_1$ the   two Lagrangians of $\Nc(T')$ associated respectively to $\nu K$ (or $N\setminus \nu K$),  and to the  Dehn filling of $Y''$. Recall that $L_0$ and $L_1$ intersect transversely at a single point. 

Let finally $\underline{L}  = \underline{L}(Y'\setminus N , c )$ be the \glag, going from  $pt$ to $pt$. The  $HSI$ homology groups of the  three manifolds are then given by: 

\begin{align*}
HSI(Y, P_{|Y},z) =& HF(  \underline{L})  \\
HSI(Y', P_{|Y'},z') =& HF(L_0, L_0,  \underline{L}) \simeq HSI(S^2\times S^1,0) \otimes HSI(Y, P_{|Y},z) \\
HSI(Y'', P_{|Y''},z'') =&  HF(L_1,L_0, \underline{L}) \simeq HSI(S^3,0) \otimes HSI(Y, P_{|Y},z)\\
 \simeq & HSI(Y, P_{|Y},z).
\end{align*}

By construction, if $[\underline{x}]\in HF(  \underline{L})$,  $F_{W_1,P_1}([\underline{x}]) =  C_+\otimes [\underline{x}]$, where $C_+$ is the generator in  degree 3. Therefore, the identity $F_{W_2,P_2 } \circ F_{W_1,P_1}= Id_{HSI(Y, P_{|Y})}$ follows from the following lemma:
\begin{lemma}$F_{W_2,P_2 }(C_+\otimes [\underline{x}]) = [\underline{x}]$. 
\end{lemma}


\begin{proof}

Recall that $F_{W_2,P_2 }$ is defined by counting quilted triangles  as in figure \ref{fig:pairofpants}. In the  present context, these triangles are equivalent to those of figure \ref{trianglebirthdeath}. Indeed, since the second symplectic manifold appearing in the sequence of Lagrangian correspondences is a point, one can  "unsew" the upper  triangle. 

\begin{figure}[!h]
    \centering
    \def\svgwidth{\textwidth}
    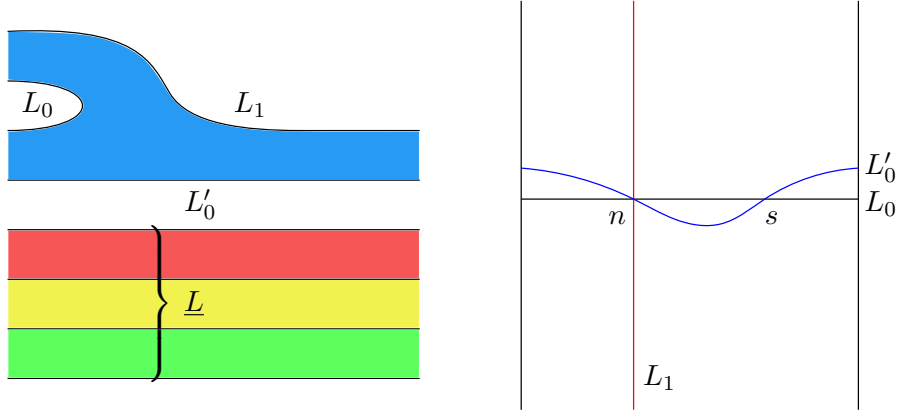
      \caption{Triangles appearing in the coefficients of $CF_{W_2,P_2 }$.}
      \label{trianglebirthdeath}
\end{figure}

Identify a tubular neighborhood  of $L_0$ in $\Nc(T')$ to a neighborhood  of the zero section in  $T^* L_0$, in such a way that $L_1$ corresponds to the fiber over a point $n\in L_0$.  Let $f\colon L_0\to \rr$ be a Morse function with  two critical points: a maximum at $n$ (for ``North'') and a minimum at some other point $s$ (for ``South''). Denote $L_0'$ the graph of $df$, and assume $f$ is small enough so that its graph is contained in the tubular neighborhood in question.

The homology $HF(L_0, L_0')$ is isomorphic to the Morse homology  of $f$, indeed there can be no Floer strips of index 1 due to the index difference of $n$ and $s$. The degree 3 generator $C_+$ thus  corresponds to the maximum $n$ of $f$ under this  identification.

Call $C\Phi\colon CF(L_1,(L_0)^T, L_0,(L_0')^T ,\underline{L} ) \to CF(L_1,(L_0')^T ,\underline{L} )$ the quilted pair-of-pant chain map, so that $CF_{W_2,P_2}( x_0, \underline{x} ) = C\Phi (n, x_0, \underline{x}) $. We shall show that $C\Phi (n,n,\underline{x}) = (n,\underline{x})$. As $L_1$ and $L_0'$ intersect transversely only at $n$, it suffices to show, for $\underline{x}, \underline{y}\in \I(\underline{L}) $,  that:

\begin{itemize}
\item if $\underline{x}\neq \underline{y}$, then $ \left( C\Phi (n,n,\underline{x}) , (n,\underline{y})\right) =0 $, and
\item if $\underline{x}= \underline{y}$, then $ \left( C\Phi (n,n,\underline{x}) , (n,\underline{y})\right) =1 $. 
\end{itemize}

Assume that there  exists a pseudo-holomorphic triangle  contributing to this coefficient. The  total index of the triangle  and the quilted strip  is zero: it follows by genericity that both have index zero (otherwise one would have negative index and would lie in a moduli space of negative dimension). Therefore the quilted strip is constant and  $\underline{x}= \underline{y}$, which proves the first point.

Assume now that $\underline{x}= \underline{y}$. The constant quilt has index zero, and is regular: indeed, for generic  perturbations,  The linearized Cauchy-Riemann  operator  associated to the  triangle is injective by  \cite[Lemma 2.27]{Seidel}, hence surjective since  the triangle has index zero. Similarly, the linearized Cauchy-Riemann  operator  associated to the quilted strip is injective, by \cite[Theorem 3.2]{WWerrata}. Any other quilt appearing in the moduli space should have same index, and thus same symplectic area by monotonicity, therefore it should be constant: the moduli space of such quilts consists  in a single point, which implies $ \left( C\Phi (n,n,\underline{x}) , (n,\underline{y})\right) =1 $.

To sum up, we have shown that  $F_{W_2,P_2 }(C_+\otimes [\underline{x}]) = [\underline{x}]$, as announced.
\end{proof}

\subsubsection{Critical point switch}

Let $Y$ be a 3-manifold. Suppose that $W_1$, a cobordism from $Y$ to $Y_1$,  corresponding to the attachment of a handle $h_1$ to $Y$, and that $W_2$, a cobordism from $Y_1$ to $Y_{12}$, correspond to the attachment of another handle $h_2$ to $Y_1$ disjoint from $h_1$. Suppose that $W_2'$ and $W_1'$ correspond  to the attachment in the opposite order, so that $W_1\cup_{Y_1} W_2 = W_2' \cup_{Y_2} W_1'$. Denote $Y_2$ the 3-manifold between  $W_2'$ and $ W_1'$, and $Y_{21}$ the other boundary of $W_1'$, as summarized in figure \ref{interversionptcrit}. The cobordism $W_1\cup_{Y_1} W_2 = W_2' \cup_{Y_2} W_1'$ is equipped with an $SO(3)$-bundle $P$, which restricts to each 4-cobordism and 3-manifold appearing, but which we will drop from the notations since there is no ambiguity.

 We will show that $F_{W_2} \circ F_{W_1} = F_{W_1'} \circ F_{W_2'}$. Four cases need to be distinguished depending on the handle dimensions: 
\begin{enumerate}
\item $h_1$ is a 1-handle and $h_2$ is a 1-handle.
\item $h_1$ is a 1-handle and $h_2$ is a 2-handle.
\item $h_1$ is a 2-handle and $h_2$ is a 2-handle.
\item $h_1$ is a 1-handle and $h_2$ is a 3-handle.
\end{enumerate}
The last remaining case (2-handle/3-handle) can be deduced from the  second one.

\begin{figure}[!h]
    \centering
    \def\svgwidth{.65\textwidth}
   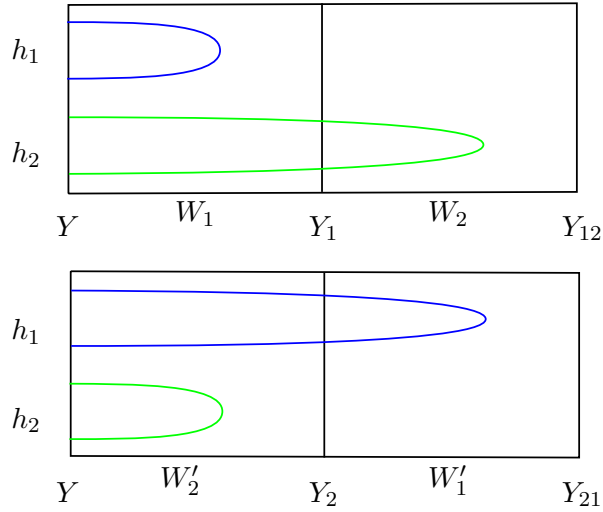
      \caption{Critical point switch.}
      \label{interversionptcrit}
\end{figure}

\vspace{.3cm}
\paragraph{\underline{First case ($h_1$ is a 1-handle and $h_2$ is a 1-handle)}}   Let $A_1$ and $A_2$ be two manifolds homeomorphic to $S^2\times S^1$, so that, with  $i=1$ or 2 and $j= 3-i$, $Y_i =  A_i\#Y$ and  $Y_{ij} = A_j\#(A_i\#Y  ) $. Let also $H_Y = HSI(Y)$ and $H_i = HSI(A_i)$, so that, according to the Künneth formula, 
\begin{align*}
HSI(Y_i) &=H_i \otimes H_Y \text{, and}\\
HSI(Y_{ij}) &= H_j\otimes H_i \otimes H_Y.
\end{align*}
Let finally $C_{+,i} \in H_i$ be the degree 3 generators. Then the composition maps 
\begin{align*}   F_{W_2}\circ F_{W_1} &\colon H_Y \to H_Y\otimes H_1 \otimes H_2,\text{ and}  \\
 F_{W_1}'\circ F_{W_2}' &\colon H_Y \to H_Y\otimes H_2 \otimes H_1
\end{align*}
are given, for $x\in H_Y$,  by: 
\begin{align*}
 F_{W_2}\circ F_{W_1}(x) =& C_{+,2}\otimes C_{+,1} \otimes x \\
 F_{W_1}'\circ F_{W_2}'(x) =& C_{+,1}\otimes C_{+,2} \otimes x .
\end{align*}
These are thus  identified via the isomorphism $HSI(Y_{12}) \simeq HSI(Y_{21}) $ induced by the identity.

\vspace{.3cm}
\paragraph{\underline{Second case ($h_1$ is a 1-handle and $h_2$ is a 2-handle)}} Keeping the previous notations, $HSI(Y_1) = H_1\otimes HSI(Y) $ and $HSI(Y_{21}) =  H_1 \otimes HSI(Y_2)$.

The claim  follows from the fact that $F_{W_2} = Id_{H_1}\otimes F_{W_2'}  $. Indeed, the last strip of the quilted surface  defining $F_{W_2}$ can be unsewed, see figure \ref{switch1handle2handle}.

\begin{figure}[!h]
    \centering
    \def\svgwidth{.45\textwidth}
    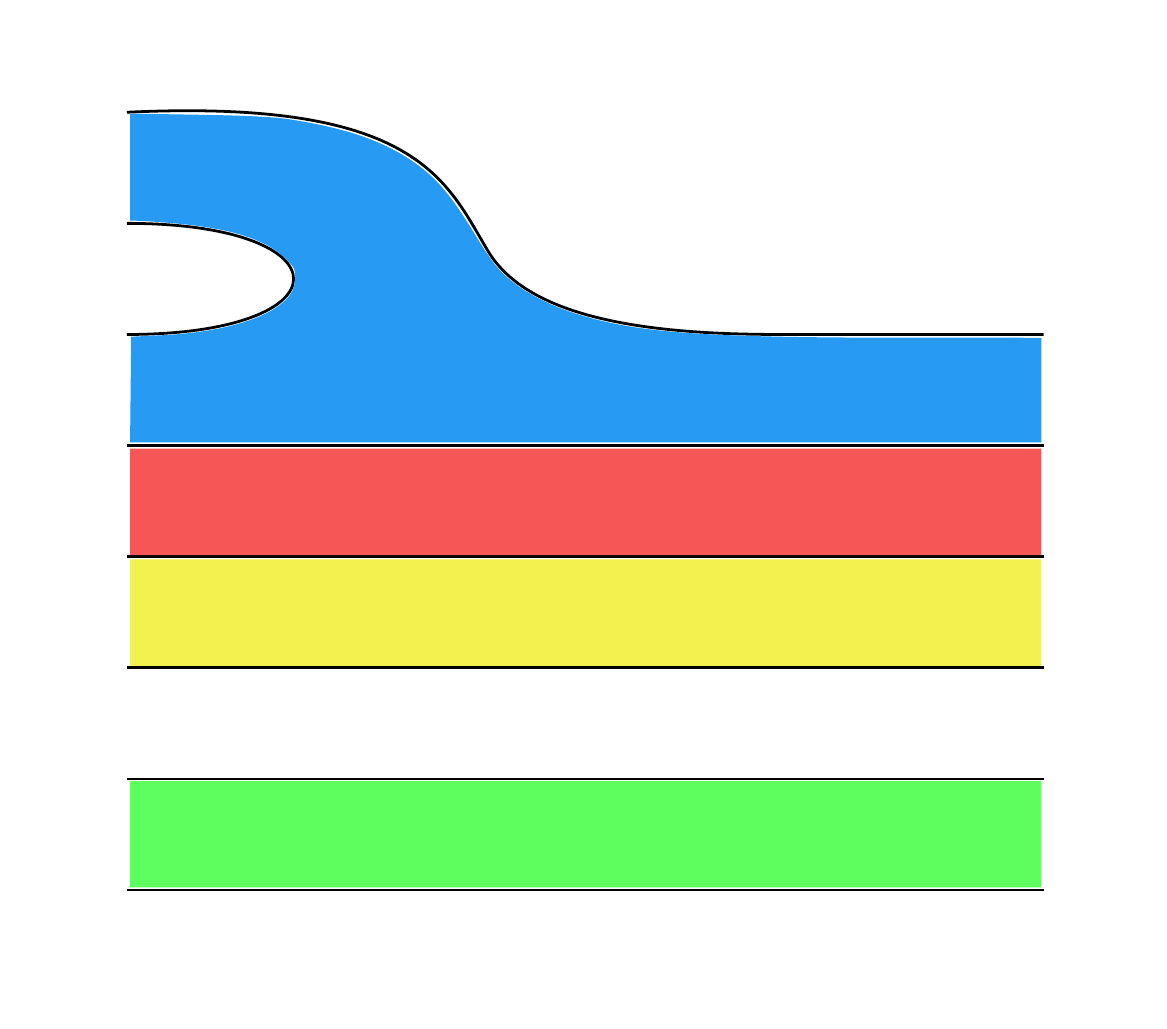
      \caption{1-handle and 2-handle switch.}
      \label{switch1handle2handle}
\end{figure}

\vspace{.3cm}
\paragraph{\underline{Third  case ($h_1$ is a 2-handle and $h_2$ is a 2-handle)}}

 Let  $K_1, K_2 \subset Y$ be two framed knots, to which we attach  the handles $h_1$ and $h_2$ respectively. Let  $N_1$ and $N_2$ be  tubular neighborhoods, $T_1$ and $T_2$ their boundaries, $M = Y\setminus \left( K_1 \cup K_2 \right)$ their complement, seen as a cobordism from $T_1$ to $T_2$. 

Denote by $\underline{L}$ the \glag\ from $\Nc(T_1)$ to $\Nc(T_2)$ associated to $M$. Let $L_1\subset\Nc(T_1)$ and $L_2 \subset\Nc(T_2)$ be the Lagrangians corresponding to $N_1$ and $N_2$. Let finally $L_1'\subset\Nc(T_1)$ and $L_2' \subset\Nc(T_2)$ be the Lagrangians corresponding to the  two Dehn fillings, so that we have:
\begin{align*}
HSI(Y) &= HF(L_1, \underline{L}, L_2), \\
HSI(Y_1) &= HF(L_1', \underline{L}, L_2), \\
HSI(Y_2) &= HF(L_1, \underline{L}, L_2'), \\
HSI(Y_{12}) &= HF(L_1', \underline{L}, L_2').
\end{align*}
The maps associated to $W_1$, $W_2$, $W_1'$ and $W_2'$ are then defined by counting quilted triangles  as in figure \ref{switch2handles}. This is by definition for $F_{W_1}$ and $F_{W_1'}$, and follows from the two following observations for  $F_{W_2}$ and $F_{W_2'}$. On the one hand, an oriented manifold  $Y$ described by a handle decomposition 
\[ Y = Y_1\cup_{\Sigma_1} Y_2 \cup_{\Sigma_2} \cdots \cup_{\Sigma_{k-1}} Y_k \]
can also be described by the reverse  decomposition 
\[ Y = Y_k\cup_{\overline{\Sigma}_{k-1}}\cdots \cup_{\overline{\Sigma}_2} Y_2 \cup_{\overline{\Sigma}_1} Y_1  ,\]
where the surfaces have their opposite orientations. Their corresponding moduli spaces  are then endowed with the opposite symplectic form. On the other hand, a pseudo-holomorphic quilt with values in a family of symplectic manifolds corresponds to its mirror image  with values in the family of symplectic manifolds with opposite symplectic forms. 

The compositions $F_{W_2}\circ F_{W_1}$ and $F_{W_1'}\circ F_{W_2'}$ are hence associated to the corresponding glued surfaces: they thus coincide, due to the homotopy suggested in figure \ref{switch2handles}.

\begin{figure}[!h]
    \centering
    \def\svgwidth{\textwidth}
    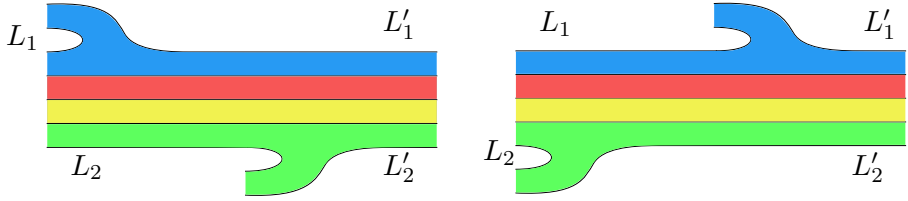
      \caption{2-handle switch.}
      \label{switch2handles}
\end{figure}

\vspace{.3cm}
\paragraph{\underline{Fourth case ($h_1$ is a 1-handle and $h_2$ is a 3-handle)}} As in the first case, let $A_1$ and $A_2$ both be manifolds diffeomorphic to $S^2\times S^1$, so to have 
\begin{align*}
 Y &\simeq A_2\#Y_2, \\
 Y_1 &\simeq A_1 \# Y\simeq A_1\# A_2\# Y_2, \\
 Y_{12} &\simeq Y_{21}\simeq A_1 \# Y_2 .
\end{align*}
Denote $H_{Y_2} = HSI(Y_2)$, $H_i = HSI(A_i)$, and let $C_{+,i}, C_{-,i} \in H_i$  be the generators  with  degrees  3 and 0 respectively.

We now check that $F_{W_2}\circ F_{W_1} = F_{W'_1}\circ F_{W'_2}$: let $x\in H_{Y_2}$, one has on the one hand: 
\begin{align*} F_{W_2}\circ F_{W_1} (C_{+,2}\otimes x) &= F_{W_2} (C_{+,1} \otimes C_{+,2}\otimes x) \\
&= C_{+,1} \otimes x \\
&=F_{W'_1}( x ) \\
&=F_{W'_1}\circ F_{W'_2}(C_{+,2} \otimes x ),
\end{align*}
and on the other hand:
\[F_{W_2}\circ F_{W_1} (C_{-,2}\otimes x) = F_{W'_1}\circ F_{W'_2}(C_{-,2}\otimes x) = 0 ,\]
which completes the proof of the independence of the decomposition.
\arnaque

Hence the cobordism maps $F_{W,P,\gamma}$ are well-defined. The following composition formula is a straightforward consequence of their construction (and  corresponds a vertical composition of 2-morphisms in the 2-category  $\textbf{Cob}_{2+1+1}$):

\begin{prop}\label{prop:compovertic}(Vertical composition) Let $(W,\gamma) = (W_1,\gamma_1) \cup (W_2,\gamma_2)$ be a composition of two cobordisms, $P$ an $SO(3)$-bundle over $W$, restricting respectively to $P_1$ and $P_2$ on  $W_1$ and $W_2$. Then,
\[ F_{W,P,\gamma} = F_{W_2,P_2,\gamma_2}\circ F_{W_1,P_1,\gamma_1}. \]
\end{prop}
\arnaque

\subsection{Maps in the surgery exact triangle}\label{flechestri}

In this section we prove that the mod 2  analogues  of two of the maps in the  exact sequence of \cite[Theorem~1.3]{surgery} can be interpreted as maps induced by cobordisms.

 Let $Y$ be a 3-manifold with torus boundary, $\alpha, \beta,\gamma$ three curves in $T = \partial Y$ satisfying $ \alpha. \beta = \beta.\gamma = \gamma.\alpha = -1$,  and $Y_\alpha$,  $Y_\beta$, $Y_\gamma$ the manifolds obtained by  Dehn filling. If $\delta \neq \mu \in \lbrace \alpha, \beta, \gamma\rbrace$, let $W_{\delta, \mu }$ be the cobordism from $Y_\delta$ to $Y_\mu$ corresponding to the attachment of a 2-handle along $\delta$ with framing $\mu$.
 
 We think of $W_{\delta, \mu }$ as a union  $W_{\delta, \mu } = Y\times [0,1] \cup_{T\times[0,1]} X_{\delta, \mu } $, where $X_{\delta, \mu } = D^2\times D^2$, where we embed $T\times[0,1]$ as a thinckening of the corner $\partial D^2\times \partial D^2$, in such a way that $\delta \times \lbrace 0 \rbrace \subset T\times[0,1] $ and $\mu \times \lbrace 1 \rbrace \subset T\times[0,1] $ bound a disc in $\partial X_{\delta, \mu } \setminus ( T\times[0,1])$. Fix also a basepoint $z$ in $T$, and let all the base paths be $\lbrace z\rbrace \times [0,1]$.

Fix an $SO(3)$-bundle $P_Y$ over $Y$, together with a trivialization of $P_Y$ on $\partial Y$. $P_Y\times [0,1]$ is then trivialized over $T\times[0,1]$. Let $P_{X_{\delta, \mu }}$  denote the trivial $SO(3)$-bundle over $X_{\delta, \mu }$. We will glue it to $P_Y\times [0,1]$ along $T\times [0,1]$ using a transition function $\tau_{\delta, \mu }\colon T\times [0,1] \to SO(3)$:
\begin{itemize}
\item define $\tau_{\beta, \gamma }$ as constantly equal to $1$.
\item in order to define $\tau_{\alpha,\beta}$, identify $T$ with $\alpha \times \beta \simeq (\rr/ \zz)^2$, and let $a$, $b$ denote coordinates on  $\alpha$ and $\beta$ respectively. thinking of $SO(3)$ as $SU(2)/\pm 1$, and $SU(2)$ as unit
quaternions, define  $\tau_{\alpha,\beta}(a,b,s) = \left[ e^{i\pi a}\right] $.
\item define $\tau_{\gamma, \alpha}$ in a similar way: this time identify $T$ with $\gamma\times \alpha$, with coordinates $c$ and $a$, and set  $\tau_{\gamma, \alpha}(c,a,s) = \left[ e^{i\pi a}\right] $.
\end{itemize}
Denote then $P_{\delta, \mu }$ the bundle over $X_{\delta, \mu }$ obtained, and $P_\alpha$, $P_\beta$, $P_\gamma$, the three bundles over  $Y_\alpha$,  $Y_\beta$ and $Y_\gamma$ respectively, so that $P_{\delta, \mu }$ is an $SO(3)$-cobordism from $P_{\delta }$ to $P_{\mu }$.

%
%
%
%

In this setting, \cite[Theorem~1.3]{surgery} says that the three groups 
\[
HSI(Y_\alpha ,P_\alpha),\ HSI(Y_\beta,P_\beta)\text{ and }HSI(Y_\gamma,P_\gamma)
\]
fit into a long exact sequence. The proof of this theorem involves (an adaptation to the $HSI$ setting of) an exact sequence of \WW, which is a quilted analog of Seidel's exact sequence for symplectic Dehn twists. It is built from a  sequence (where CSI stands for the chain complex defining HSI):

\[\xymatrix{  CSI(Y_\alpha ,P_\alpha)  \ar[r]^{C\Phi_1}&  CSI(Y_\beta,P_\beta) \ar[r]^{C\Phi_2}&  CSI(Y_\gamma,P_\gamma)  ,
}\]
where the maps $C\Phi_1$ and $C\Phi_2$ are defined respectively by counting quilted pair of pants, and quilted sections of a Lefschetz fibration.

We briefly describe the quilted Lefschetz fibration involved  in the definition of $C\Phi_2$, and refer to \cite[Sec.~5.1.5]{surgery} for more details. Let $\underline{L} = ( L_{01}, \cdots, L_k)$ stand for the \glag\ from $M_0$ to $pt$ associated with a Cerf decomposition of $Y$. $L_0$ and $S$ are the two Lagrangian 3-spheres of $M_0$ corresponding respectively to $P_\gamma$ and $P_\alpha$, and denoting $\tau_S$ the generalized Dehn twist around $S$, $\tau_S L_0$ then corresponds to $P_\beta$. The base of the quilted Lefschetz fibration is a quilted strip as in figure~\ref{fig:PHI_2}, where each patch $P_1$, $P_2$, $\cdots$ except the first $P_0$, is decorated with a trivial fibration $P_i\times M_i$. The first patch $P_0$ (in blue in figure~\ref{fig:PHI_2}) is decorated with a Lefschetz fibration with one critical value at the cross, whose generic fiber is $M_0$, and whose vanishing cycle is the Lagrangian sphere $S$.
The holonomy around this critical value corresponds to the generalized Dehn twist around $S$. As in \cite{Seidel}, we trivialize this Lefschetz fibration on the complement of the vertical slit joining the critical value and the top boundary component. The Lagrangian boundary and seam conditions are then understood in this trivialization.

\begin{figure}[!h]
    \centering
    \def\svgwidth{.7\textwidth}
    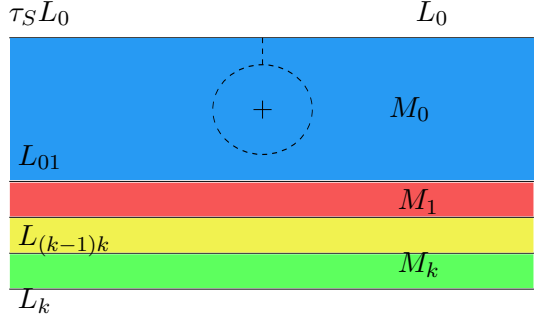
      \caption{Quilted Lefschetz fibration defining $C\Phi_2$.}
      \label{fig:PHI_2}
\end{figure}

\begin{theo}\label{thinterpfleches} 
The morphisms $\Phi_1$ and $\Phi_2$ induced in  homology with  $\Z{2}$-coefficients  by $C\Phi_1$ and $C\Phi_2$ coincide with  the maps $F_{W_{\alpha\beta}, P_{\alpha\beta}}$ and $F_{W_{\beta\gamma}, P_{\beta\gamma} }$.
\end{theo}

The third morphism of the  exact sequence, constructed as a connecting  homomorphism, might not a priori be induced by $W_{\gamma\alpha}$. However, the same argument than the one used by  Lisca and Stipsicz for Heegaard-Floer   homology   \cite[Section 2]{LiscaStipsicz} implies: 
\begin{cor}\label{corinterpfleches} The maps $F_{W_{\alpha\beta}, P_{\alpha\beta}}$, $F_{W_{\beta\gamma}, P_{\beta\gamma} }$, and $F_{W_{\gamma\alpha}, P_{\gamma\alpha}}$ form an  exact sequence:
\[ \xymatrix{  HSI(Y_\beta,P_\beta)\ar[rr]^{F_{W_{\beta\gamma}, P_{\beta\gamma} }} & & HSI(Y_\gamma,P_\gamma)\ar[ld]^{\ \ \ \ \ F_{W_{\gamma\alpha}, P_{\gamma\alpha} }}  \\ & HSI(Y_\alpha,P_\alpha )\ar[lu]^{F_{W_{\alpha\beta}, P_{\alpha\beta}}}  & ,}\]
where the groups are with  coefficients in $\Z{2}$.
\end{cor}

\begin{proof}[Proof of corollary \ref{corinterpfleches}] We use the cyclic symmetry of the surgery triad. Denote $L_\delta^\epsilon = \lbrace\mathrm{Hol}_\delta = \epsilon I \rbrace$ the Lagrangian in $\N((\partial Y)')$, and $\underline{L} = \underline{L}(Y, P_Y)$, so that:
\begin{align*}
HSI(Y_\alpha,P_\alpha )&= HF(L_\alpha^-, \underline{L}) , \\
HSI(Y_\beta,P_\beta)&=   HF(L_\beta, \underline{L}),   \\
HSI(Y_\gamma,P_\gamma)&= HF(L_\gamma, \underline{L}) .
\end{align*}

To prove \cite[Theorem~1.3]{surgery}, one  applies the Dehn twist exact sequence \cite[Theorem~5.2]{surgery} to $L_0 = L_\gamma$ and $S=L_\alpha^-$. It follows that $\mathrm{Ker}(F_{W_{\beta\gamma}, P_{\beta\gamma} }) = \mathrm{Im}(F_{W_{\alpha\beta}, P_{\alpha\beta} })$.

If one proceeds similarly with  $L_0 = L_\alpha^-$ and $S= L_\beta$, one gets a similar exact sequence  involving  $F_{W_{\beta\gamma}, P_{\beta\gamma} }$ and $F_{W_{\gamma\alpha}, P_{\gamma\alpha} }$.  It follows that $\mathrm{Ker}(F_{W_{\gamma\alpha}, P_{\gamma\alpha} }) = \mathrm{Im}(F_{W_{\beta\gamma}, P_{\beta\gamma} })$.

Finally, the choice $L_0 = L_\beta$ and $S=L_\gamma$ allows one to prove  
\[
\mathrm{Ker}(F_{W_{\alpha\beta}, P_{\alpha\beta} }) = \mathrm{Im}(F_{W_{\gamma\alpha}, P_{\gamma\alpha} })
,\]
 hence  the   announced exact sequence.
\end{proof}

\begin{proof}[Proof of theorem~\ref{thinterpfleches}]

First, $\Phi_1 = F_{W_{\alpha\beta}, P_{\alpha\beta}}$  by definition of the map $F_{W_{\alpha\beta}, P_{\alpha\beta}}$.

By pushing the critical value of the Lefschetz fibration to  the upper boundary  of the quilted surface, and then stretching the surface  (see figure \ref{contrac}), one can see  that $C\Phi_2$ is homotopic to the contraction of the pair-of-pant product  
\[
CF(L_0, \tau_S L_0 )\otimes CF(\tau_S L_0 , \underline{L}) \to CF( L_0 , \underline{L})
\]
with the cocycle $ c_{S,L_0} \in CF(L_0, \tau_S L_0 )$ defined by the  Lefschetz fibration specified in the figure. Indeed, this is a standard argument as in for example \cite[sec.~5.1.6]{surgery}, involving a parametrized moduli space.

It remains to notice that $ c_{S,L_0}$ coincides with  the generator $C$ used to define  $F_{W_{\beta\gamma}, P_{\beta\gamma}}$. This is due to the fact that $L_0$ and  $\tau_S L_0$ intersect transversely at a single point $x$, for which there exists a unique index 0 section, the  "constant" one (thinking of the fibration as trivial near that point), which is regular according to \cite[Lemma 2.27]{Seidel}. Indeed, by monotonicity, every other section of index zero would also have zero area and consequently be constant, therefore it is the only such section.
 Indeed, since $\dim S = 3\geq 2$, and since the generic fiber is monotone and simply connected, it follows from \cite[Prop.~4.10]{WWtriangle} that the fibration is monotone.
\end{proof}

\begin{figure}[!h]
    \centering
    \def\svgwidth{.7\textwidth}
    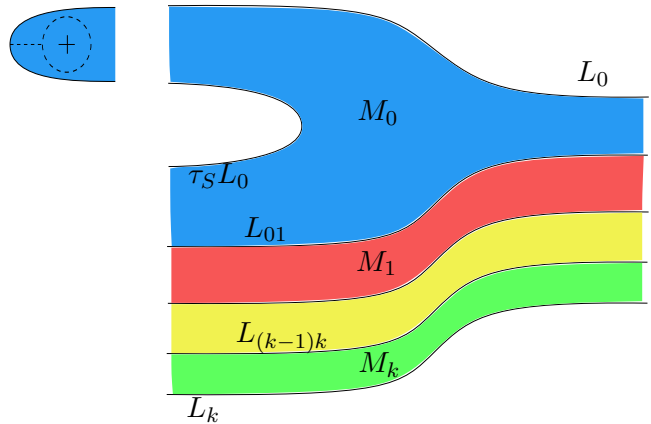
      \caption{Contraction of a quilted pair-of-pant with  a cocycle.}
      \label{contrac}
\end{figure}

\subsection{Examples and elementary properties}

\begin{prop}\label{cepetwo}Let $W = \cc P^2 \setminus \lbrace \text{two balls} \rbrace$. Since $H^2(W;\Z{2}) \simeq \Z{2}$, there are two isomorphism classes of $SO(3)$-bundles over $W$. Let $P$ stand for the trivial bundle, and $P'$ for one non-trivial one.
\begin{enumerate}
\item $F_{W,P} = F_{W,P'} = 0 $.
\item Denoting $\overline{W}$ the cobordism with opposite orientation, $F_{\overline{W},P}$ is an isomorphism, and $F_{\overline{W},P'} = 0 $.
\end{enumerate}
\end{prop}

\begin{proof}
Recall that $W$ corresponds to the attachment of a 2-handle to $S^3$ along  the trivial knot, with  framing 1.
\begin{enumerate}
\item  The surgery exact sequence applied  to the triad corresponding to the surgeries $\infty$, 1,2  on the trivial knot is of the form, with $Q$ being either $P$ or $P'$, and denoting $S^3_\alpha$ the surgery $\alpha$ on the  trivial knot:
\[ \xymatrix{  HSI(S^3_{1})  \simeq \Z{2} \ar[rr] & & HSI(S^3_{2}) \simeq \Z{2}^2 \ar[ld] \\ & {HSI(S^3 _\infty ) \simeq \Z{2}} \ar[lu]^{F_{W,Q} }  & ,  } \]
which implies the claimed result.

\item Consider now the  triad corresponding to the surgeries $\infty$, -1,0  on the trivial knot. If one endows $S^3_0$  with a nontrivial bundle $P_{S^3_0}$, one obtains:
\[ 
\xymatrix{  HSI(S^3_{-1})  \simeq \Z{2} \ar[rr] & & HSI(S^3_{0},P_{S^3_0}) = 0 \ar[ld] \\ & {HSI(S^3 _\infty ) \simeq \Z{2}} \ar[lu]^{F_{\overline{W},P} }  & ,  } 
\]
and in the two other cases, one obtains:
\[ \xymatrix{  HSI(S^3_{-1})  \simeq \Z{2} \ar[rr] & & HSI(S^3_{0}) \simeq \Z{2}^2 \ar[ld] \\ & {HSI(S^3 _\infty ) \simeq \Z{2}} \ar[lu]^{F_{\overline{W},P'} }  & .  } \]
\end{enumerate}

\end{proof}

We introduce the following  invariant  for closed 4-manifolds:
\begin{defi}Let $X$ be a closed  4-manifold, and $P$ an $SO(3)$-bundle over $X$. Denote $W$ the manifold $X$ with two open balls removed, seen as a cobordism from $S^3$ to $S^3$. Its corresponding map 
\[
F_{W,P_{|W}}\colon HSI(S^3) \to HSI(S^3)
\]
is a multiplication by some number in $\Z{2}$, which we denote $\Psi_{X,P}$. 
\end{defi}


The following proposition \ref{prop:compohoriz}  describes the effect of a "horizontal"  composition  of cobordisms (which  corresponds a horizontal composition of 2-morphisms in the 2-category  $\textbf{Cob}_{2+1+1}$).
\begin{defi} Let  $W$ and $W'$ be two cobordisms, going respectively from $Y_1$ to $Y_2$, and from $Y_1'$ to $Y_2'$. Let $l$, $l'$ be two paths in $W$ and $W'$ connecting the two boundaries. One can form their \emph{horizontal composition} by removing tubular neighborhoods  of $l$ and $l'$, and gluing together the two remaining pieces along the boundaries of the removed pieces. One obtains a new cobordism $W\#_{vert}W'$ from $Y_1\# Y_1'$ to $Y_2\# Y_2' $. If $P$ and $P'$ are $SO(3)$-bundles over $W$ and $W'$ respectively, one can use the trivialization along the base paths to connect them to an $SO(3)$-bundle $P\#_{vert}P'$ over $W\#_{vert}W'$ 
\end{defi}

\begin{prop}\label{prop:compohoriz} (Horizontal composition)
Let $W\#_{vert}W'$ be a vertical composition as before. Then, under the following identifications (\cite[Th~1.1]{surgery}, recall that the groups are with coefficients in $\Z{2}$, hence there are no  torsion summands):
\begin{align*}
 HSI (Y_1\# Y_1', P_1 \# P_1') &= HSI (Y_1, P_1 ) \otimes HSI (Y_1',P_1') \\ HSI (Y_2\# Y_2', P_2 \# P_2') &= HSI (Y_2, P_2) \otimes HSI (Y_2', P_2'),\\
\end{align*}
The map associated to $W\#_{vert}W'$ is given by:
\[ 
F_{W\#_{vert}W',P\#P'} = F_{W,P}\otimes F_{W',P'} .
\]
In particular, if $X$ is a closed manifold and $W\#X$ stands for the connected sum at an interior point, 
\[ 
F_{W\#X,P_W \# P_X} = \Psi_{X,P_X} \cdot F_{W,P_W}. 
\]
\end{prop}

\begin{proof} Take two  decompositions of   $W$ and $W'$ in elementary cobordisms, this  induces a decomposition for $W\#_{vert}W'$. The fact that there is a point in the middle of the \glag\ guarantees that the maps coming from the  handles of $W$ don't interact with  those of $W'$:  at the chain level, the induced map corresponds to $CF_{W,P}\otimes CF_{W',P'}$. The result at the level of homology groups follows then from the naturality of the Künneth formula, see \cite[Theorem 11.10.2]{tomDieck}.  
\end{proof}

From this result and  propositon \ref{cepetwo}, it follows:

\begin{cor}\label{eclatement} 
Let $W$ be a 4-cobordism, 
\begin{enumerate}
\item $F_{W\# \cc P^2,P} =  0 $, for any $SO(3)$-bundle $P$
\item If $P$ is nontrivial in restriction to $\overline{\cc P}^2$, then 
\[
F_{W\# \overline{\cc P}^2,P} =  0 ,
\]
otherwise $F_{W\# \overline{\cc P}^2,P} =  F_{W,P_{|W}} $.
\end{enumerate}
\end{cor}

\begin{remark}We observe here a slight   difference with Heegaard-Floer homology: if $(Y_\alpha, Y_\beta, Y_\gamma)$ is a surgery triad and $W_{\delta \gamma}$ denotes the handle attachment 4-cobordisms, then the three morphisms associated to $W_{\delta \gamma}$, endowed with the trivial bundles, don't form an exact sequence. Indeed, for the triad $(S^3_0, S^3_1, S^3_\infty)$ associated to the  trivial knot in the sphere, the cobordism going from  $S^3_1$ to $S^3_\infty$ induces an isomorphism, even though $HSI(S^3_0) = HSI(S^2\times S^1)$ has rank 2.
\end{remark}

\bibliographystyle{alpha}
\bibliography{biblio}

\end{document}